\documentclass[final,leqno]{siamltex704}
\usepackage{hyperref}
\usepackage{amsmath}
\usepackage{graphicx}
\usepackage{array}
\usepackage{lineno}
\usepackage[notcite,notref]{showkeys}
\usepackage{tikz}
\usepackage{amsfonts,amssymb}
\usepackage{dsfont}
\usepackage{pifont}
\usepackage{subfigure}

\newtheorem{remark}{Remark}[section]
\renewcommand{\ldots}{\dotsc}

\newtheorem{model-problem}{Problem}

\newcommand{\br}{\textbf{r}}

\newcommand{\bv}{\textbf{v}}

\def\SS{{\mathcal S}}

\def\W{{\mathcal W}}

\def\S{{\mathcal Q_{h}}}

\def\dt{D_{\boldsymbol{\tau}}}
\def\L{{\mathcal L}}

\def\bn{{\bf n}}

\def\3bar{{|\hspace{-.02in}|\hspace{-.02in}|}}

\def\bbq{\begin{equation*}}
\def\eeq{\end{equation*}}
\def\br{\begin{eqnarray}}
\def\er{\end{eqnarray}}
\def\brr{\begin{eqnarray*}}
\def\err{\end{eqnarray*}}

\def\O{\Omega}

\def\grad{\nabla}
\def\pa{\partial}

\def\bn{{\bf n}}

\def\3bar{{|\hspace{-.02in}|\hspace{-.02in}|}}

\def\bmu{{\boldsymbol\mu}}

\newtheorem{PD-WG}{SIMPLIFIED PRIMAL-DUAL WEAK GALERKIN ALGORITHM}

\begin{document}

\setlength{\parindent}{0.25in} \setlength{\parskip}{0.08in}

\title{A Simplified primal-dual weak Galerkin finite element method for Fokker-Planck Type Equations}
\author{
Dan Li\thanks{School of Mathematics and Statistics, Northwestern Polytechnical University, Xi'an, Shannxi 710072, China.} \and
Chunmei Wang \thanks{Department of Mathematics \& Statistics, Texas Tech University, Lubbock, TX 79409, USA (chunmei.wang@ttu.edu). The research of Chunmei Wang was partially supported by National Science Foundation Award DMS-1849483.}
}
\maketitle

\begin{abstract}
A simplified primal-dual weak Galerkin (S-PDWG) finite element method is designed for the Fokker-Planck type equation with non-smooth diffusion tensor and drift vector. The discrete system resulting from S-PDWG method has significantly fewer degrees of freedom compared with the one resulting from the PDWG method proposed by Wang-Wang \cite{WW-fp-2018}. Furthermore, the condition number of the S-PDWG method is smaller than the PDWG method \cite{WW-fp-2018} due to the introduction of a new stabilizer, which provides a potential for designing fast algorithms. Optimal order error estimates for the S-PDWG approximation are established in the $L^2$ norm. A series of numerical results are demonstrated to validate the effectiveness of the S-PDWG method.
\end{abstract}



\section{Introduction}
We consider the Fokker-Planck type model problem with homogeneous Dirichlet boundary condition which seeks $u$ such that
\begin{equation}\label{model-problem}
\begin{split}
\grad\cdot(\bmu u)-\frac{1}{2}\sum_{i,j=1}^{d}\pa_{ij}^{2}(a_{ij}u)&=f, \quad \mbox{in}~~ \O,\\
u&=0,\quad \mbox{on}~~ \pa\O ,
\end{split}
\end{equation}
where $\Omega$ is an open bounded domain in $\mathbb{R}^d$ $(d=2, 3)$ with Lipschitz continuous boundary $\pa \Omega$ and $f \in L^2(\Omega)$.   We assume that the diffusion tensor $a(x)=\{a_{ij}(x)\}_{d\times d}\in [L^{\infty}(\Omega)]^{d\times d}$ is symmetric and uniformly positive definite in the domain $\O$, and the drift vector is $\pmb{\mu}\in [L^{\infty}(\O)]^{d}$. Note that  the diffusion tensor $a(x)$ and the drift vector $\pmb{\mu}$ are non-smooth functions so that the exact solution $u$ possesses discontinuities. Throughout this paper, we assume that  the diffusion tensor $a(x)$ and the drift vector $\pmb{\mu}$ are piecewise smooth functions in the domain $\Omega$.

The Fokker-Planck problem arises in science and technology such as statistics, physics, engineering and biological system \cite{Fokker-1914,Gardiner-1985,Perthame2007,Risken1989}.
The Fokker-Planck equation has been numerically solved by finite difference methods \cite{F-p_PZ2018}, finite element methods  \cite{BW1996,BS1968,KN-2006,f-p_DG1985,f-p_FEM1991}, discontinuous Galerkin methods \cite{f-p_DG2016}, spectral methods \cite{MS-2018} and  primal-dual weak Galerkin method \cite{WW-fp-2018}.  
For an efficient implementation, researchers are devoted to designing numerical methods to decrease the degrees of freedom of the discrete system. \cite{ellip_MWY2015} proposed an optimal combination for the polynomial space to reduce the number of unknowns. The Schur complement was proposed in \cite{ellip_MWYZ2015,ellip_MWY2017}, where the unknowns of the numerical system were only defined on the element boundary. \cite{WY-ellip_MC2014,WY_ZZ-2015,MawellLI,WW_divcur-2016} developed a technique to employ the tangential component and/or normal component of the unknown vector in the numerical systems.
A stabilizer-free weak Galerkin finite element method was developed in \cite{ellip_XYZ2020}.

The PDWG method has been successfully applied to several challenging problems including the second order elliptic equation in non-divergence form \cite{PDWG_MC2017}, Fokker-Planck equation \cite{WW-fp-2018}, the elliptic Cauchy problem \cite{ellipCau_WW2020}, the first-order transport problem \cite{wwhyperbolic}, the linear convection equation \cite{Lwwhyperbolic}, and the div-curl system \cite{LW-divcurl-2020}.
The essential idea of the PDWG method is to formulate the original equation into a constraint optimization problem in the weak Galerkin framework \cite{wybasis,ellip_JCAM2013}. The similar framework was also proposed  by Burman \cite{Burman2013, burman2014} which was named stabilized finite element methods. In the S-PDWG method, the gradient vector is first decomposed into the tangential and normal components, and only the normal component is counted into the degrees of freedom. This leads to a significant reduction of computational complexity of the S-PDWG numerical scheme compared with  the PDWG method proposed by Wang-Wang \cite{WW-fp-2018}. In addition, the newly-introduced stabilizer in S-PDWG is reversible such that the condition number  resulting from the S-PDWG numerical scheme is relatively small. This feature provides a potential for designing effective fast algorithms. This work is a non-trivial extension of PDWG method for solving the Fokker-Planck equation developed by Wang-Wang \cite{WW-fp-2018}.

The paper is organized as follows. In Section \ref{Preliminaries}, we present a weak formulation for the Fokker-Planck problem \eqref{model-problem} and its dual problem. Section \ref{Section:WeakGradient-definition} briefly reviews the definition of weak differential operators and their discrete versions. In Section \ref{Section:numerical scheme}, we propose a S-PDWG  algorithm for  the Fokker-Planck problem \eqref{model-problem}. Section \ref{Section:well-posed} is devoted to establishing the inf-sup condition and the solution existence and uniqueness of the S-PDWG scheme. The error equations for the S-PDWG scheme are derived in Section \ref{Section:error equation}. In Section \ref{Section:error estimate for projections}, the technical estimates are derived which play an important role in deriving the error estimates in Section \ref{Section:error estimate}. Section \ref{Section:error estimate} is dedicated to establishing an optimal order of error estimates in the $L^2$ norm. Finally, a series of numerical results are reported to demonstrate the effectiveness of the  S-PDWG method.

\section{Weak Formulations}\label{Preliminaries}
We shall follow the standard definitions for the Sobolev spaces and norms \cite{Ciarlet-2002,Gilbarg-1983}. Let $D\subset\mathbb{R}^d$ be any bounded domain with Lipschitz continuous boundary $\pa D$. We shall use $(\cdot,\cdot)_{s,D}$, $\|\cdot\|_{s,D}$ and $|\cdot|_{s,D}$ to denote the inner product, norm and semi-norm for the Sobolev space $H^s(D)$ for any $s\geq 0$, respectively. For $D=\O$, we shall drop the corresponding subscript $D$ in the semi-norm, norm and inner product notations. For $s=0$, the space $H^0(D)$ coincides with $L^2(D)$; the semi-norm, norm and inner product are denoted by $|\cdot|_D$, $\|\cdot\|_D$ and $(\cdot, \cdot)_D$, respectively.  We use ``$\lesssim$'' to denote ``less than or equal to up to a general constant independent of the meshsize or functions in the inequality''.

Denote by $\pmb{\L}: H^2(\O) \to  L^2(\O)$ the bounded linear operator defined by
$$\pmb{\L} v=\bmu\cdot \nabla v+\frac{1}{2}\sum_{i,j=1}^{d}a_{ij}\pa_{ji}^{2}v.$$ A weak formulation for the Fokker-Planck model problem \eqref{model-problem} seeks a function $u\in L^2(\O)$ satisfying $u=0$ on $\pa \O$ such that
 \begin{equation}\label{weak formulation}
(u,\pmb{\L} v)=-(f,v),\quad \forall v\in H^2(\O)\cap H_{0}^{1}(\O).
\end{equation}

For the well-posedness of the weak formulation \eqref{weak formulation}, we assume the second order differential operator $\pmb{\L}$ satisfies the $H^2$-regularity property in the sense that there exists an unique strong solution $\Phi\in H^2(\O)\cap H_{0}^{1}(\O)$ such that
\begin{equation}\label{regurality operator}
\pmb{\L}\Phi=\chi,\quad \|\Phi\|_2\lesssim \|\chi\|,
\end{equation}
for any given $\chi\in L^2(\O)$ \cite{Talenti-1965, WW-fp-2018}.

The dual problem for this weak formulation \eqref{weak formulation} reads: Find $\rho\in H^2(\O)\cap H_{0}^{1}(\O)$ such that
\begin{equation}\label{dual formulation}
(w,\pmb{\L}\rho)=0,\quad \forall w\in L^2(\O).
\end{equation}
The $H^2$-regularity assumption (\ref{regurality operator}) for the differential operator $\pmb{\L}$ indicates that the dual problem \eqref{dual formulation} has one and only one trivial solution $\rho=0$.

\section{Weak Partial Derivatives and Discrete Weak Partial Derivatives}\label{Section:WeakGradient-definition}

The goal of this section is to review the definitions of weak second order partial derivative and weak gradient, as well as their corresponding discrete versions \cite{WW_bihar-2014, ellip_JCAM2013}.

Let $T$ be any polygonal or polyhedral domain with boundary $\pa T$. A weak function $v$ on $T$ refers to a triplet $\{v_{0},v_{b},\pmb{v_{g}}\}$
such that $v_{0}\in L^{2}(T)$, $v_{b}\in L^{2}(\pa T)$ and $\pmb{v_{g}}\in [L^{2}(\pa T)]^{d}$. The first and second components $v_{0}$
and $v_{b}$ represent the value of $v$ in the interior and on the boundary of $T$, respectively; the third component $\pmb{v_{g}}=[v_{g1},...,v_{gd}]'$ intends to represent the value of $\nabla v$ on $\pa T$. Note that $v_b$ and $\pmb{v_{g}}$ are not necessarily related to the traces of $v_0$ and $\nabla v_0$ on $\pa T$, respectively. Denote by ${\W}(T)$ the space of all weak functions on $T$; i.e.,
\begin{eqnarray*}\label{weak function space}
{\W}(T)=\{v=\{v_{0},v_{b},\pmb{v_{g}}\}:v_{0}\in L^{2}(T),v_{b}\in L^{2}(\pa T),\pmb{v_{g}}\in[L^{2}(\pa T)]^{d}\}.
\end{eqnarray*}
Let $P_{r}(T)$ be the space of polynomials on the element $T$ with degree no more than $r$.

\begin{definition}\cite{WW_bihar-2014}(Weak second order partial derivative)
The weak second order partial derivative of any weak function $v\in {\W}(T)$, denoted by $\pa_{ij,w}^{2}v$ ($i, j=1, \cdots, d$), is defined as a bounded linear functional in the dual space of $H^{2}(T)$ such that
\begin{eqnarray}\label{weak partial}
(\pa_{ij,w}^{2}v,\varphi)_{T}=:(v_{0},\pa_{ji}^{2}\varphi)_{T}-\langle v_{b}n_{i},\pa_{j}\varphi\rangle_{\pa T}+\langle v_{gi},\varphi n_{j}\rangle_{\pa T},
\end{eqnarray}
for any $\varphi\in H^{2}(T)$. Here, $\pmb{n}=(n_{1},...,n_{d})'$ is the unit outward normal direction to $\pa T$.
\end{definition}
\begin{definition}\cite{ellip_JCAM2013}(Weak gradient)
The weak gradient of any weak function $v\in {\W}(T)$, denoted by $\nabla_{w}v$, is defined as a linear functional in the dual space of $[H^{1}(T)]^{d}$
such that
\begin{eqnarray}\label{weak gradient}
(\nabla_{w}v,\pmb{\psi})_{T}=:-(v_{0},\nabla\cdot\pmb{\psi})_{T}+\langle v_{b},\pmb{\psi}\cdot\mathbf{n}\rangle_{\pa T},\quad\forall\pmb{\psi}\in [H^{1}(T)]^{d}.
\end{eqnarray}
\end{definition}

\begin{definition}\cite{WW_bihar-2014} (Discrete weak second order partial derivative)
The discrete weak second order partial derivative of any weak function $v\in {\W}(T)$, denoted by $\pa_{ij,w,r,T}^{2}v$, is defined as a unique polynomial in $P_{r}(T)$ satisfying
\begin{eqnarray}\label{discrete weak partial derivetive}
(\pa_{ij,w,r,T}^{2}v,\varphi)_{T}=(v_{0},\pa_{ji}^{2}\varphi)_{T}-\langle v_{b}n_{i},\pa_{j}\varphi \rangle_{\pa T}+\langle v_{gi},\varphi n_{j}\rangle_{\pa T},\quad\forall \varphi\in P_{r}(T).
\end{eqnarray}
\end{definition}
Applying the usual integration by parts to \eqref{discrete weak partial derivetive} yields
\begin{equation}\label{discrete weak partia derivetive l1-2}
\begin{split}
(\pa_{ij, w, r, T}^{2}v,\varphi)_{T}=(\pa_{ij}^{2}v_{0},\varphi)_{T}+\langle (v_{0}-v_{b})n_{i},\pa_{j}\varphi \rangle_{\pa T}
-\langle \pa_{i}v_{0}-v_{gi},\varphi n_{j}\rangle_{\pa T}.
\end{split}
\end{equation}
\begin{definition}\cite{ellip_JCAM2013} (Discrete weak gradient)
The discrete weak gradient of any weak function $v\in {\W}(T)$, denoted by $\nabla_{w,r,T}v$, is defined as a unique polynomial vector in $[P_{r}(T)]^{d}$ such that
\begin{eqnarray}\label{discrete weak gradient1-1}
(\nabla_{w,r,T}v,\boldsymbol{\psi})_{T}=-(v_{0},\nabla\cdot\boldsymbol{\psi})_{T}+\langle v_{b}, \boldsymbol{\psi}\cdot\mathbf{n}\rangle_{\pa T},\quad \forall\boldsymbol{\psi}\in [P_{r}(T)]^{d}.
\end{eqnarray}
\end{definition}
Applying the usual integration by parts to \eqref{discrete weak gradient1-1} gives rise to
\begin{eqnarray}\label{discrete weak gradient1-2}
(\nabla_{w,r,T}v,\boldsymbol{\psi})_{T}=(\nabla v_{0},\boldsymbol{\psi})_{T}-\langle v_{0}-v_{b},\boldsymbol{\psi}\cdot\boldsymbol{n}\rangle_{\pa T},\quad \forall\boldsymbol{\psi}\in [P_{r}(T)]^{d}.
\end{eqnarray}

\section{Simplified Primal-Dual Weak Galerkin Methods}\label{Section:numerical scheme}

Let ${\cal T}_{h}$ be a finite element partition of the domain $\O$ into polygons in $2D$ or polyhedra in $3D$ which is shape regular as described in \cite{WY-ellip_MC2014}. Denote by $\mathcal{E}_h$ the set of all edges/flat faces in ${\cal T}_{h}$, by $\mathcal{E}_h^{0}$ the set of all interior edges/faces of ${\cal T}_{h}$. The meshsize of ${\cal T}_{h}$ is $h=\max_{T\in{\cal T}_{h}}h_{T}$ with $h_T$ being the meshsize of the element $T\in {\cal T}_h$.

Note that $\nabla v$ can be decomposed into its tangential and normal components; i.e.,
\begin{eqnarray*}\label{vg}
\begin{split}
\nabla v=(\nabla v\cdot\bn)\bn +\bn\times(\nabla v\times\bn).
\end{split}
\end{eqnarray*}
We introduce a set of assigned unit normal vectors on $e\in\mathcal{E}_h$, denoted by $N_{h}$; i.e.,
$$
N_{h}=\{\mathbf{n_{e}}:\mathbf{n_{e}} \;is\;an \; assigned \;normal\;vector \;to\; e,\; \forall \ e\in\mathcal{E}_h\}.
$$
Similarly, $\pmb{v_{g}}$ can be decomposed into its tangential and normal components; i.e.,
$$
\pmb{v_{g}}=v_n \bn_e+{\dt}v_{b} \pmb{\tau}.
$$
where $v_n=\nabla v\cdot\bn_e$, ${\dt}v_{b}$ is the value of $\nabla v_b$ in the tangential direction $\pmb{\tau}$ with $\pmb{\tau}=(\tau_1,\ldots,\tau_d)'$ being the tangential vector to $e\subset \pa T$, and $\bn_{e}\in N_{h}$.

For any integer $k\geq1$, the local weak function space on each element $T$ is given by
\begin{equation*}
\begin{split}
V_{k}(T)=\{\{v_{0},v_{b},\pmb{v_{g}}=&v_{n}\bn_{e}+{\dt}v_{b} \pmb{\tau}\},v_{0}\in P_{k}(T),v_{b} \in P_{k}(e),\\
&v_{n}\in P_{k-1}(e), \bn_{e}\in N_{h}, e\subset \pa T\}.
\end{split}
\end{equation*}
Patching $V_{k}(T)$ over all the elements $T\in {\cal T}_{h}$ through a common value $v_{b}$ on the interior edges/faces $e\in \mathcal{E}_h$ yields  a global weak finite element space $V_{h, k}$; i.e.,
$$V_{h,k}=\{v=\{v_{0},v_{b},\pmb{v_{g}}=v_{n}\bn_{e}+{\dt}v_{b} \pmb{\tau}\}:v|_{T}\in V_{k}(T),T\in {\cal T}_{h}\}.$$
We further introduce a subspace of $V_{h, k}$ with homogeneous boundary conditions, denoted by $V_{h, k}^{0}$; i.e.,
$$V_{h, k}^{0}=\{v=\{v_{0},v_{b},\pmb{v_{g}}=v_{n}\bn_{e}+{\dt}v_{b} \pmb{\tau}\}\in V_{h, k}:v_{b}|_{e}=0,e\subset\pa\O\}.$$

For any given integer $s\geq0$, let $W_{h,s}$ be the finite element space given by
$$
W_{h, s}=\{w:w|_{T}\in P_{s}(T),s=k-1~\text{or}~k-2,\;T\in {\cal T}_{h}\}.
$$
When it comes to the lowest order $k=1$, the only option is $s=0$.

For the convenience of notation, denote by $\nabla_{w}$ the discrete weak gradient $\nabla_{w, k-1,T}$  computed by (\ref{discrete weak gradient1-1}) on each element $T$ with $r=k-1$; i.e.,
$$(\nabla_{w}v)|_T=\nabla_{w,k-1,T}(v|_T), \quad v\in V_{h,k}.$$
Similarly,  denote by $\pa_{ij,w}^{2}v$ the discrete weak second order partial derivative $\pa_{ij,w,s,T}^{2}$ computed by \eqref{discrete weak partial derivetive} on each element $T$ with $r=s$; i.e.,
$$(\pa_{ij,w}^{2}v)|_T=\pa_{ij,w,s,T}^{2}(v|_T),\quad v\in V_{h,k}.$$
The differential operator $\pmb{\L}$ is thus discretized by
$$\pmb{\L}_{w}(v)=\pmb{\mu}\cdot\nabla_{w}v+\frac{1}{2}\sum_{i,j=1}^{d}a_{ij}\pa_{ji,w}^{2}v.$$

On each element $T\in\mathcal{T}_h$, we introduce two bilinear forms as follows
\begin{eqnarray*}
{\SS}(\rho,\sigma)&=&\sum_{T\in\mathcal{T}_h}s_{T}(\rho,\sigma)+\sum_{T\in\mathcal{T}_h}c_{T}(\rho,\sigma),\quad \forall  \rho,\sigma\in V_{h, k},\\
 b(\sigma,v)&=&\sum_{T\in\mathcal{T}_h}b_{T}(\sigma,v),\qquad\qquad \qquad\forall \sigma\in V_{h,k},v\in W_{h,s},
\end{eqnarray*}
where
\begin{equation}\label{stabilizer}
\begin{split}
s_{T}(\rho,\sigma)=& h_{T}^{-3}\int_{\pa T} (\rho_{0}-\rho_{b})(\sigma_{0}-\sigma_{b})ds\\
&+h_{T}^{-1}\int_{\pa T} (\nabla\rho_{0}\cdot\bn_{e}-\rho_{n})(\nabla\sigma_{0}\cdot\bn_{e}-\sigma_{n})ds+\gamma_1\int_{T}\pmb{\L}\rho_0\pmb{\L}\sigma_0dT,
\end{split}
\end{equation}
\begin{equation}\label{additionc}
\begin{split}
c_{T}(\rho,\sigma)=\gamma_2\int_{T}\nabla\sigma_0\nabla \rho_0dT+\gamma_3\sum_{i,j=1}^{d}\int_{T}\pa_{ij}^{2}\sigma_0\pa_{ij}^{2} \rho_0dT,
\end{split}
\end{equation}
\begin{eqnarray*}
b_{T}(\sigma,v)&=&(v,\pmb{\L}_{w}(\sigma))_{T}\label{bilinear form b},
\end{eqnarray*}
with $\gamma_i \geq 0$ $(i=1,2,3)$ being given parameters.

\begin{PD-WG} The numerical scheme for the Fokker-Planck problem \eqref{model-problem} is as follows: Find $(u_{h};\rho_{h})\in W_{h,s}\times V_{h,k}^{0}$ such that $u_h=0$ on $\pa \O$ satisfying
\begin{eqnarray}
{\SS}(\rho_{h},\sigma)+b(\sigma,u_{h})&=&-(f,\sigma_{0}),\qquad \forall \sigma\in V_{h,k}^{0},\label{PDWG-scheme1}\\
b(\rho_{h},v)&=&0,\qquad \qquad\, \quad\forall v\in W_{h,s}\label{PDWG-scheme2}.
\end{eqnarray}
\end{PD-WG}

\begin{remark} 
For the case of piecewise smooth diffusion tensor $a(x)$, $\gamma_1>0$ is required; for the case of piecewise constant diffusion tensor $a(x)$, $\gamma_1=0$ is a feasible option. For the case of $|\gamma_2|+|\gamma_3|> 0$, the stabilizer $c(\cdot, \cdot)$ is reversible and provides a relatively small condition number for the S-PDWG scheme. Therefore, S-PDWG is advantageous in designing fast algorithms. 
\end{remark}

\section{Solution Existence, Uniqueness and Stability}\label{Section:well-posed}
In this section,  the solution existence and uniqueness of the S-PDWG  scheme \eqref{PDWG-scheme1}-\eqref{PDWG-scheme2} will be demonstrated through Babu\u{s}ka-Brezzi theory \cite{babuska, Brezzi-1974}.

For each element $T\in{\cal T}_{h}$, denote by $Q_{0}$ the $L^{2}$ projection onto $P_{k}(T)$. On each edge or face $e\subset\mathcal{E}_h$, denote by $Q_{b}$ and $Q_{g}$  the $L^2$ projections onto $P_{k}(e)$ and $P_{k-1}(e)$, respectively. For any $w\in H^{2}(\O)$, denote by $Q_{h}w\in V_{h, k}$ the $L^{2}$ projection  such that on each element $T$,
\begin{equation*}\label{projection-operator}
\begin{split}
&Q_{h}w=\{Q_{0}w,Q_{b}w, Q_{g}((\nabla w\cdot\bn_{e}) \bn_{e} +  \bn_{e}\times(\nabla w\times\bn_{e}))\}.
\end{split}
\end{equation*}
Denote by ${\S}^{(s)}$ the $L^{2}$ projection onto $P_s(T)$.

\begin{lemma}\label{commutative properties}
The following commutative properties hold true
\begin{eqnarray}
\pa_{ij,w}^{2}(Q_{h}w)&=&{\S}^{(s)}(\pa_{ij}^{2}w),\quad i,j=1,\ldots,d,\label{commutative pro-1}\\
\nabla_{w}(Q_{h}w)&=&{\S}^{(k-1)}(\nabla w),\label{commutative pro-2}
\end{eqnarray}
for any $w\in H^{2}(T)$.
\end{lemma}
\begin{proof}
It follows from the integration by parts and \eqref{weak partial} that
\begin{equation*}
\begin{split}
&(\pa_{ij, w}^{2}(Q_{h}w),\varphi)_{T}\\
=&(Q_{0}w,\pa_{ji}^{2}\varphi)_{T}-\langle Q_{b}w,n_{i}\pa_{j}\varphi\rangle_{\pa T}+\langle  Q_{g}((\nabla w\cdot\bn_{e}) \bn_{e} +   \bn_{e}\times(\nabla w\times\bn_{e}))_i ,\varphi n_{j}\rangle_{\pa T}\\
=&(w,\pa_{ji}^{2}\varphi)_{T}-\langle w,n_{i}\pa_{j}\varphi\rangle_{\pa T}+\langle ((\nabla w\cdot\bn_{e}) \bn_{e} + \bn_{e}\times(\nabla w\times\bn_{e}))_i,\varphi n_{j}\rangle_{\pa T}\\
=&(w,\pa_{ji}^{2}\varphi)_{T}-\langle w n_{i},\pa_{j}\varphi \rangle_{\pa T}+\langle (\nabla w\cdot\bn_{e}) (\bn_{e})_i,\varphi n_{j}\rangle_{\pa T}\\
&+\langle \pa_{i}w-((\nabla w\cdot\bn_{e})\bn_{e})_i,\varphi n_{j}\rangle_{\pa T}\\
=&(\pa_{ij}^{2}w,\varphi)_T\\
=&({\S}^{(s)}(\pa_{ij}^{2}w),\varphi)_T,
\end{split}
\end{equation*}
for any $\varphi\in P_{s}(T)$, where we used the identity $\bn_e\times(\nabla w\times\bn_e)=\nabla w-(\nabla w\cdot\bn_e)\bn_e$ and the notation $(\cdot)_i$ denotes the $i-$th component of a vector. This completes the proof of \eqref{commutative pro-1}.

\eqref{commutative pro-2} can be proved in a similar fashion, and the details can be found in \cite{WW_HWG-2015,WY_stokes-2016}.
\end{proof}

For any $\sigma\in V_{h,k}$, we define two seminorms as follows
\begin{eqnarray}
\3bar\sigma\3bar_{w}^2&=&s(\sigma,\sigma), \label{stabilizer definition-1}\\
\3bar\sigma\3bar_{c}^2&=&c(\sigma,\sigma). \label{stabilizer definition-2}
\end{eqnarray}

Now, we are in a position to verify the solution existence and uniqueness of the S-PDWG scheme \eqref{PDWG-scheme1}-\eqref{PDWG-scheme2}. It is easy to check that the boundedness and coercivity of the bilinear forms ${\SS}(\cdot,\cdot)$ and $b(\cdot,\cdot)$ hold true.

\begin{lemma}(Inf-Sup Condition)\label{inf-sup condition}
Assume that the drift vector  $\pmb{\mu}\in [L^{\infty}(\O)]^{d}$ and the  diffusion tensor $a(x)$ is uniformly piecewise continuous with respect to the finite element partition ${\cal T}_{h}$. There exists a constant $\beta>0$ such that for any $v\in W_{h,s}$, there exists a weak function $\tilde{\sigma}\in V_{h, k}^{0}$ such that
\begin{eqnarray*}
b(\tilde{\sigma},v)&\geq&\frac{1}{2}\|v\|^{2},\label{inf-sup-1}\\
\3bar\tilde{\sigma}\3bar_{w}^2+\3bar\tilde{\sigma}\3bar_{c}^2&\leq&\beta\|v\|^2,\label{inf-sup-2}
\end{eqnarray*}
provided that the meshsize $h$ satisfies $h \leq h_{0}$ for a sufficiently small, but fixed $h_{0}>0$.
\end{lemma}
\begin{proof}
The proof is similar to the proof of Lemma $5.3$ in \cite{WW-fp-2018}. The details are thus omitted here.
\end{proof}

\begin{theorem}\label{solution-exist}
Assume that the drift vector  $\pmb{\mu}\in [L^{\infty}(\O)]^{d}$ and the diffusion tensor $a(x)\in [L^{\infty}(\O)]^{d\times d}$ are uniformly piecewise continuous with respect to the finite element partition ${\cal T}_{h}$. 
The S-PDWG scheme \eqref{PDWG-scheme1}-\eqref{PDWG-scheme2} has one and only one solution if the meshsize satisfies $h \leq h_0$  for a sufficiently small, but fixed $h_{0}>0$.
\end{theorem}
\begin{proof}
It suffices to prove the system \eqref{PDWG-scheme1}-\eqref{PDWG-scheme2} with homogeneous data $f=0$ has a unique trivial solution $u_h=0$ and $\rho_h=0$. To this end, letting $\sigma=\rho_{h}$ in \eqref{PDWG-scheme1} and $v=u_{h}$ in \eqref{PDWG-scheme2} gives ${\SS}(\rho_{h},\rho_{h})=0$. This further yields $\sum_{T\in {\cal T}_h}s_T(\rho_{h},\rho_{h})=0$ and $\sum_{T\in {\cal T}_h}c_T(\rho_{h},\rho_{h})=0$. As to the case that $\gamma_1\geq0$, $\gamma_2>0$ and $\gamma_3\geq 0$, it follows from \eqref{stabilizer}-\eqref{additionc} that
$\rho_{0}-\rho_{b}=0$ and $\rho_{n}-\nabla\rho_{0}\cdot\bn_{e}=0$ on each $\pa T$; and $\nabla \rho_0=0$ on each $T\in {\cal T}_h$, which gives $\rho_0\in C^0(\Omega)$ and further $\rho_0\equiv const$ in $\Omega$. Using $\rho_b=0$ on $\partial \Omega$ and  $\rho_{0}-\rho_{b}=0$ on each $\pa T$, we have $\rho_0\equiv 0$ and further $\rho_h\equiv 0$ in $\Omega$. For the case that $\gamma_1\geq0$, $\gamma_2\geq0$ and $\gamma_3> 0$, we have $\partial_{ij}^2\rho_0=0$ in each $T$ for $i, j =1,\cdots, d$. Since $\rho_{0}-\rho_{b}=0$ and $\rho_{n}-\nabla\rho_{0}\cdot\bn_{e}=0$ on each $\pa T$, we have $\rho_0\in C^1(\Omega)$. Therefore, $\Delta \rho_0=0$ in $\Omega$ with the boundary condition $\rho_0=0$ on $\partial \Omega$ due to $\rho_h\in V_{h, k}^0$. Thus, $\rho_0\equiv 0$ in $\Omega$. From $\rho_{0}=\rho_{b}$ on each $\partial T$, we have $\rho_b\equiv 0$ and further $\rho_h\equiv 0$ in $\Omega$. As to the cases of $\gamma_1>0$, $\gamma_2\geq 0$, $\gamma_3\geq0$, the proof to show $\rho_h\equiv 0$ in $\Omega$ can be found in Theorem 5.4 in \cite{WW-fp-2018}.

Substituting $\rho_h=0$ into \eqref{PDWG-scheme1} yields
\begin{equation*}\label{exist-uniqueness-1}
b(\sigma, u_{h})=0,\qquad\forall\sigma\in V_{h,k}^{0}.
\end{equation*}
From the inf-sup condition in Lemma \ref{inf-sup condition}, there exists a weak function $\tilde{\sigma}\in V_{h,k}^{0}$ such that
\begin{equation*}
b(\tilde{\sigma},u_{h})\geq\frac{1}{2}\|u_h\|^2,
\end{equation*}
which gives $u_h\equiv 0$ in $\Omega$. This completes the proof of the theorem.
\end{proof}

\section{Error Equations}\label{Section:error equation}
The goal of this section is to derive the error equations for the S-PDWG numerical scheme \eqref{PDWG-scheme1}-\eqref{PDWG-scheme2}. Note that the error equations are critical to establish the error estimates in Section \ref{Section:error estimate}.

Let $u$ and $(u_h;\rho_h)\in W_{h,s}\times V_{h, k}^0$ be the solutions of the model problem \eqref{model-problem} and the S-PDWG algorithm \eqref{PDWG-scheme1}-\eqref{PDWG-scheme2}, respectively. Note that $\rho_h$ approximates the trivial exact solution $\rho=0$.
We define the error functions $e_{h}$ and $\epsilon_{h}$ as follows
\begin{eqnarray}
e_{h} &=& u_{h}-{\S}^{(s)}u, \label{EE-1}\\
\epsilon_{h} &=& \rho_{h}-Q_{h}\rho=\rho_{h}.\label{EE-2}
\end{eqnarray}

\begin{lemma}\label{auxi-equation}
For any $\sigma\in V_{h,k}$ and $v\in W_{h, s}$, there holds
\begin{equation*}\label{est1:07:23}
b_T(\sigma,v)=(\pmb{\L}\sigma_{0},v)_{T}+R_{T}(\sigma,v),
\end{equation*}
where
\begin{equation}\label{EE-21:43}
\begin{split}
R_{T}(\sigma,v)=&\frac{1}{2}\sum_{i,j=1}^{d}\langle(\sigma_{0}-\sigma_{b})n_{j},\pa_{i}({\S}^{(s)}(a_{ij}v))\rangle_{\pa T}\\
&-\frac{1}{2}\sum_{i, j=1}^{d}\langle\pa_{j}\sigma_{0}-(\sigma_{n}  (\bn_e)_j+ {\dt}\sigma_{b}\boldsymbol{\tau}_j),{\S}^{(s)}(a_{ij}v)n_{i}\rangle_{\pa T}\\
&-\langle\sigma_{0}-\sigma_{b},{\S}^{(k-1)}(\pmb{\mu}v)\cdot\bn\rangle_{\pa T}.
\end{split}
\end{equation}
\end{lemma}
\begin{proof}
From \eqref{discrete weak gradient1-2} and \eqref{discrete weak partia derivetive l1-2}, one has
\begin{equation*}
\begin{split}
b_T(\sigma,v)=&(\pmb{\L}_{w}(\sigma),v)_{T}\\
=&(\pmb{\mu}\cdot \nabla_{w}\sigma,v)_{T}+\frac{1}{2}\sum_{i, j=1}^{d}(a_{ij}\pa_{ji,w}^{2}\sigma,v)_{T}\\
=&(\nabla_{w}\sigma,{\S}^{(k-1)}(\pmb{\mu} v))_{T}+\frac{1}{2}\sum_{i,j=1}^{d}(\pa_{ji, w}^{2}\sigma,{\S}^{(s)}(a_{ij}v))_{T}\\
=&(\nabla\sigma_0,{\S}^{(k-1)}(\pmb{\mu} v))_{T}+\frac{1}{2} \sum_{i,j=1}^{d}(\pa_{ji}^{2}\sigma_0,{\S}^{(s)}(a_{ij}v))_{T}\\
&-\langle\sigma_{0}-\sigma_{b},{\S}^{(k-1)}(\pmb{\mu}v)\cdot\bn\rangle_{\pa T}
+\frac{1}{2}\sum_{i, j=1}^{d}\langle(\sigma_{0}-\sigma_{b})n_{j},\pa_{i}({\S}^{(s)}(a_{ij}v))\rangle_{\pa T}\\
&-\frac{1}{2}\langle\pa_{j}\sigma_{0}-(\sigma_{n} (\bn_e)_j+ {\dt}\sigma_{b} {\boldsymbol{\tau}}_{j}),{\S}^{(s)}(a_{ij}v)n_{i}\rangle_{\pa T}\\
=&(\nabla\sigma_{0},\pmb{\mu}v)_{T}+ \frac{1}{2}\sum_{i,j=1}^{d}(\partial_{ji}^{2}\sigma_{0},a_{ij}v)_{T}+R_{T}(\sigma,v)\\
=&(\pmb{\L}\sigma_{0},v)_T+R_{T}(\sigma,v),
\end{split}
\end{equation*}
where  $R_{T}(\sigma,v)$ is given in \eqref{EE-21:43}.

This completes the proof of the Lemma.
\end{proof}
\begin{lemma}
Let $e_h$ and $\epsilon_h$ be the error functions defined by \eqref{EE-1}-\eqref{EE-2}, respectively. The following error equations hold true
\begin{eqnarray}
\SS(\epsilon_{h},\sigma) +b(\sigma,e_{h})&=&\zeta_{u}(\sigma),\quad\forall \sigma\in V_{h,k}^{0},\label{error-equation-1}\\
b(\epsilon_{h},v)&=&0,  \quad\quad~~~ \forall v\in W_{h,s},\label{error-equation-2}
\end{eqnarray}
where the term $\zeta_{u}(\sigma)$ is given by
\begin{equation}\label{error equation-3}
\begin{split}
&\zeta_{u}(\sigma)= \sum_{T\in {\cal T}_h}\langle\sigma_{b}-\sigma_{0},(\pmb{\mu}u-{\S}^{(k-1)}(\pmb{\mu}{\S}^{(s)}u))\cdot\bn\rangle_{\pa T}\\
&+\frac{1}{2}\sum_{i,j=1}^{d}\langle\sigma_{0}-\sigma_{b}, \pa_{i}(a_{ij}u-{\S}^{(s)}(a_{ij}{\S}^{(s)}u))n_{j} \rangle_{\pa T}\\
&-\frac{1}{2}\sum_{i,j=1}^{d}\langle\pa_{j}\sigma_{0}-(\sigma_{n} (\bn_e)_j+  {\dt} \sigma _{b} \boldsymbol{\tau}_{j}),\Big(a_{ij}u-{\S}^{(s)}(a_{ij}{\S}^{(s)}u)\Big)n_{i}\rangle_{\pa T}\\
&-({\L}\sigma_{0},{\S}^{(s)}u-u)_{T}.
\end{split}
\end{equation}
\end{lemma}
\begin{proof}
Using \eqref{EE-2} and \eqref{PDWG-scheme2} gives
$$b(\epsilon_{h},v)=b(\rho_{h},v)=0,\qquad\forall v\in W_{h,s},$$
which completes the proof of \eqref{error-equation-2}.

As to \eqref{error-equation-1}, it follows from \eqref{EE-1}-\eqref{EE-2} and \eqref{PDWG-scheme1} that
\begin{equation}\label{error equation-4}
\begin{split}
&\SS(\epsilon_{h},\sigma)+b(\sigma,e_{h})\\
 =&-(f,\sigma_{0})-b(\sigma,{\S}^{(s)}u)\\
 =&-(f,\sigma_{0})- \sum_{T\in {\cal T}_h}\{(\pmb{\L}\sigma_{0},{\S}^{(s)}u)_T+R_{T}(\sigma,{\S}^{(s)}u)\}\\
 =&-(f,\sigma_{0})- \sum_{T\in {\cal T}_h}\{(\pmb{\L}\sigma_{0},u)_T+(\pmb{\L}\sigma_{0},{\S}^{(s)}u-u)_T+R_{T}(\sigma,{\S}^{(s)}u)\},
\end{split}
\end{equation}
where we used Lemma \ref{auxi-equation} by letting $v={\S}^{(s)}u \in W_{h, s}$.

As to the term $\sum_{T\in {\cal T}_h}(\pmb{\L}\sigma_{0},u)_T$ on the last line of \eqref{error equation-4}, we have
\begin{equation}\label{Error-equation-6}
\begin{split}
  &  \sum_{T\in {\cal T}_h}(\pmb{\L}\sigma_{0},u)_{T}\\
=&\sum_{T\in {\cal T}_h}(\pmb{\mu}\cdot \nabla \sigma_{0},u)_T+\frac{1}{2}\sum_{i,j=1}^{d}(a_{ij}\pa_{ji}^{2}\sigma_{0},u)_{T}\\
=&\sum_{T\in {\cal T}_h}(-\sigma_{0},\nabla\cdot(\pmb{\mu}u))_{T}+\langle\sigma_{0},\pmb{\mu}u\cdot\bn\rangle_{\pa T}
+\frac{1}{2}\sum_{i,j=1}^{d}(\sigma_{0},\pa_{ij}^{2}(a_{ij}u))_{T}\\&-\frac{1}{2}\langle\sigma_{0}n_{j},\pa_{i}(a_{ij}u)\rangle_{\pa T} +\frac{1}{2}\langle\pa_{j}\sigma_{0},a_{ij}u  n_{i}\rangle_{\pa T}\\
=&-(\sigma_{0},f)+\sum_{T\in {\cal T}_h}\langle\sigma_{0},\pmb{\mu}u\cdot\bn\rangle_{\pa T}-\frac{1}{2}\sum_{i,j=1}^{d}\langle\sigma_{0},\pa_{i}(a_{ij}u)n_{j}\rangle_{\pa T}\\
& +\frac{1}{2}\langle\pa_{j}\sigma_{0},a_{ij}u  n_{i}\rangle_{\pa T}\\
=&-(\sigma_{0},f)+\sum_{T\in {\cal T}_h}\langle\sigma_{0}-\sigma_b,\pmb{\mu}u\cdot\bn\rangle_{\pa T}-\frac{1}{2}\sum_{i,j=1}^{d}\langle\sigma_{0}-\sigma_b,\pa_{i}(a_{ij}u)n_{j}\rangle_{\pa T}\\
& +\frac{1}{2}\langle\pa_{j}\sigma_{0}-(\sigma_{n} (\bn_e)_j+{\dt}\sigma_{b}\boldsymbol{\tau}_{j}),a_{ij}u  n_{i}\rangle_{\pa T},
\end{split}
\end{equation}
where we used \eqref{model-problem}, the usual integration by parts, and the following identities
$\sum_{T\in {\cal T}_h}\langle\sigma_{b},u\pmb{\mu}\cdot\bn\rangle_{\pa T}=0$, $\sum_{T\in {\cal T}_h}\sum_{i,j=1}^{d}\frac{1}{2}\langle\sigma_{b},\pa_{i}(a_{ij}u)  n_{j}\rangle_{\pa T}=0$, and
$\sum_{T\in {\cal T}_h}\sum_{i,j=1}^{d}\frac{1}{2}\langle \sigma_{n} (\bn_e)_j+{\dt}\sigma_{b}\boldsymbol{\tau}_{j},a_{ij}u  n_{i}\rangle_{\pa T}=0$, due to the facts that $\sigma_b=0$ and $u=0$ on $\partial \Omega$. Therefore, substituting \eqref{Error-equation-6} into \eqref{error equation-4} gives \eqref{error-equation-1}.

This completes the proof of the Lemma.
\end{proof}

\section{Technical Estimates}\label{Section:error estimate for projections}
We shall present the following technical results which are critical  to derive the error estimates for the numerical solution arising from the S-PDWG algorithm \eqref{PDWG-scheme1}-\eqref{PDWG-scheme2} in Section \ref{Section:error estimate}.
\begin{lemma}\cite{WY-ellip_MC2014, WW-fp-2018}\label{Section:error estimate for projections-1}
Let ${\cal T}_h$ be a finite element partition of $\O$ satisfying the shape regular assumptions as specified in \cite{WY-ellip_MC2014}. For  $0\leq t \leq \min\{2,k\}$,  there holds
\begin{eqnarray*}
\sum_{T\in {\cal T}_h} h_{T}^{2t}\|u-Q_{0}u\|_{t,T}^{2} &\lesssim& h^{2(m+1)}\|u\|_{m+1}^{2}, \quad m\in[t-1,k] ,\quad k\geq1, \\
\sum_{T\in {\cal T}_h} h_{T}^{2t}\|u-{\S}^{(k-1)}u\|_{t,T}^{2} &\lesssim& h^{2m}\|u\|_{m}^{2}, \quad m\in[t,k] ,\quad k\geq1, \\
\sum_{T\in {\cal T}_h} h_{T}^{2t}\|u-{\S}^{(k-2)}u\|_{t,T}^{2} &\lesssim& h^{2m}\|u\|_{m}^{2}, \quad m\in[t,k-1] ,\quad k\geq2.
\end{eqnarray*}
\end{lemma}

\begin{lemma}\cite{WW-fp-2018}\label{PROJECTION-ESTIMATE-2}
Let ${\cal T}_h$ be a finite element partition of $\O$ satisfying the shape regular assumptions as specified in \cite{WY-ellip_MC2014}.  Assume that the diffusion tensor $a(x)$ and the drift tensor $\pmb{\mu}$ are uniformly piecewise smooth up to order $m-1$ in $\O$ with respect to the finite element partition ${\cal T}_h$. For any $v\in \prod_{T\in {\cal T}_h}H^{m-1}(T)\cap H^{2}(T)$, there holds
\begin{eqnarray*}
\Big(\sum_{T\in {\cal T}_h} h_{T}^{3}\|\pmb{\mu}v-{\S}^{(k-1)}(\pmb{\mu}{\S}^{(s)}v)\|_{\pa T}^{2}\Big)^{\frac{1}{2}}&\lesssim& h^{m}\|v\|_{m-1},\\
\Big(\sum_{T\in {\cal T}_h}\sum_{i,j=1}^{d} h_{T}^{3}\|\pa_{i}(a_{ij}v-{\S}^{(s)}(a_{ij}{\S}^{(s)}v))\|_{\pa T}^{2}\Big)^{\frac{1}{2}}&\lesssim&
h^{m-1}(\|v\|_{m-1}+h\delta_{m,2}\|v\|_{2}),\\
\Big(\sum_{T\in {\cal T}_h}\sum_{i,j=1}^{d} h_{T}\|a_{ij}v-{\S}^{(s)}(a_{ij}{\S}^{(s)}v)\|_{\pa T}^{2}\Big)^{\frac{1}{2}}&\lesssim&h^{m-1}\|v\|_{m-1},
\end{eqnarray*}
where $m\in[2, k+1]$ if $s=k-1$ and $m\in [2, k]$ if $s=k-2$, and $\delta_{m,2}$ is the usual Kronecker's delta with value 1 when $m = 2$ and value 0 otherwise.
\end{lemma}

 \begin{lemma}\label{PROJECTION-ESTIMATE-3}
For any $\sigma\in V_{h, k}$, there holds
\begin{equation*}
\Big(\sum_{T\in {\cal T}_h}\sum_{i, j=1}^{d}h_{T}^{-1}\|\pa_{j}\sigma_{0}-(\sigma_{n}  (\bn_e)_j+{\dt}\sigma_{b}\boldsymbol{\tau}_{j})\|_{\pa T}^{2}\Big)^{\frac{1}{2}}\lesssim\3bar\sigma\3bar_{w}.
\end{equation*}
\end{lemma}
\begin{proof}
Note that $\nabla\sigma_{0}=(\nabla\sigma_{0}\cdot\bn_{e})\bn_{e}+(\nabla\sigma_{0}\cdot\pmb{\tau})\pmb{\tau}.$
From the triangle inequality, the inverse inequality and \eqref{stabilizer definition-1}, there holds
\begin{equation*}\label{PROJECTION-ESTIMATE-4}
\begin{split}
&\Big(\sum_{T\in {\cal T}_h}\sum_{i,j=1}^{d}h_{T}^{-1}\|\pa_{j}\sigma_{0}-(\sigma_{n}(\bn_e)_j+{\dt}\sigma_{b}\boldsymbol{\tau}_{j})\|_{\pa T}^{2}\Big)^{\frac{1}{2}}\\
=&\Big(\sum_{T\in {\cal T}_h}\sum_{i,j=1}^{d}h_{T}^{-1}\|(\nabla\sigma_{0}\cdot\bn_{e})(\bn_e)_j+(\nabla\sigma_{0}\cdot\pmb{\tau})\boldsymbol{\tau}_j-(\sigma_{n}(\bn_e)_j+{\dt}\sigma_{b}\boldsymbol{\tau}_{j})\|_{\pa T}^{2}\Big)^{\frac{1}{2}}\\
\leq&
\Big(\sum_{T\in {\cal T}_h}\sum_{i, j=1}^{d}h_{T}^{-1}\|(\nabla\sigma_{0}\cdot\bn_{e})(\bn_e)_j-\sigma_{n}(\bn_e)_j\|_{\pa T}^{2}\Big)^{\frac{1}{2}}\\
&+\Big(\sum_{T\in {\cal T}_h}\sum_{i,j=1}^{d}h_{T}^{-1}\|(\nabla\sigma_{0}\cdot\pmb{\tau})\boldsymbol{\tau}_j-{\dt}\sigma_{b}\boldsymbol{\tau}_{j}\|_{\pa T}^{2}\Big)^{\frac{1}{2}}\\
\lesssim&\3bar\sigma\3bar_{w}+\Big(\sum_{T\in {\cal T}_h}\sum_{i,j=1}^{d}h_{T}^{-1}\|\nabla\sigma_{0}\cdot\pmb{\tau}-{\dt}\sigma_{b}\|_{\pa T}^{2}\Big)^{\frac{1}{2}}\\
\lesssim&\3bar\sigma\3bar_{w}+\Big(\sum_{T\in {\cal T}_h}\sum_{i,j=1}^{d}h_{T}^{-1}h_{T}^{-2}\|\sigma_{0}-\sigma_{b}\|_{\pa T}^{2}\Big)^{\frac{1}{2}}\\
\lesssim&\3bar\sigma\3bar_{w}.
\end{split}
\end{equation*}

This completes the proof of the lemma.
\end{proof}

\section{Error Estimates}\label{Section:error estimate}
This section is devoted to establishing the error estimates for the S-PDWG approximation arising from the numerical scheme \eqref{PDWG-scheme1}-\eqref{PDWG-scheme2}.

\begin{theorem}\label{THM:energy-estimation}
Let $k\geq1$ and $\gamma_1>0$. Let $u$ and $(u_h;\rho_h)\in W_{h, s}\times V_{h, k}^0$ be the exact solution of the model problem \eqref{model-problem} and the numerical solution of the S-PDWG scheme \eqref{PDWG-scheme1}-\eqref{PDWG-scheme2}, respectively. Assume the diffusion tensor $a(x)$ and the drift tensor $\pmb{\mu}$ are uniformly piecewise smooth up to order $s+1$ in $\O$ with respect to the finite element partition ${\cal T}_h$. We further assume that the exact solution of the model problem \eqref{model-problem} satisfies $u\in\prod_{T\in {\cal T}_h}H^{s+1}(T)\cap H^2(T)$. The following error estimate holds true
\begin{equation}\label{Error-error-estimate-1}
\begin{split}
\3bar\epsilon_{h}\3bar_{w}+\|e_h\|\lesssim  h^{k}\|u\|_{k-1}+  h^{1+s}((1+\gamma_1^{-1/2})\|u\|_{s+1}+h\delta_{k,2}\|u\|_{2}),
\end{split}
\end{equation}
provided that $h<h_0$ for a sufficiently small, but fixed $h_{0}>0$.
\end{theorem}
\begin{proof} Note that $s=k-2$ or $s=k-1$. Letting $\sigma=\epsilon_{h}\in V_{h, k}^0$ in the error equation \eqref{error-equation-1} and from \eqref{error-equation-2}, we have
\begin{equation}\label{Error-error-estimate-2}
\SS(\epsilon_{h},\epsilon_{h}) =\zeta_{u}(\epsilon_{h}).
\end{equation}

For any $\sigma\in V_{h, k}^0$, using the Lemmas \ref{Section:error estimate for projections-1}-\ref{PROJECTION-ESTIMATE-3}, \eqref{error equation-3}, and the Cauchy-Schwarz inequality, there holds
\begin{equation}\label{Error-error-estimate-5-5}
\begin{split}
  & \qquad |\zeta_u(\sigma)|\\
 & \lesssim\Big(\sum_{T\in {\cal T}_h}h_{T}^{-3}\|\sigma_{b}-\sigma_{0}\|_{\pa T}^{2}\Big)^{\frac{1}{2}}\Big(\sum_{T\in {\cal T}_h}h_{T}^{3}\|\pmb{\mu}u-{\S}^{(k-1)}(\pmb{\mu}{\S}^{(s)}u)\|_{\pa T}^{2}\Big)^{\frac{1}{2}}\\
&+\Big(\sum_{T\in {\cal T}_h}h_{T}^{-3}\|\sigma_{b}-\sigma_{0}\|_{\pa T}^{2}\Big)^{\frac{1}{2}}\\
&\cdot\Big(\sum_{T\in {\cal T}_h}\sum_{i, j=1}^{d}h_{T}^{3}\| \pa_{i}(a_{ij}u-{\S}^{(s)}(a_{ij}{\S}^{(s)}u))\|_{\pa T}^{2}\Big)^{\frac{1}{2}}\\
&+\Big(\sum_{T\in {\cal T}_h}\sum_{j=1}^{d}h_{T}^{-1}\|\pa_{j}\sigma_{0}-(\sigma_{n} (\bn_e)_j+{\dt}\sigma_{b} \boldsymbol{ {\tau}}_{j})\|_{\pa T}^{2}\Big)^{\frac{1}{2}}\\
&\cdot \Big(\sum_{T\in {\cal T}_h}\sum_{i,j=1}^{d}h_{T}\|a_{ij}u-{\S}^{(s)}(a_{ij}{\S}^{(s)}u)\|_{\pa T}^{2}\Big)^{\frac{1}{2}}\\
&+\Big(\sum_{T\in {\cal T}_h}\gamma_1\|{\L}\sigma_{0}\|_{T}^{2}\Big)^{\frac{1}{2}}\Big(\sum_{T\in {\cal T}_h}\gamma_1^{-1}\|{\S}^{(s)}u-u\|_{T}^{2}\Big)^{\frac{1}{2}}\\
\lesssim& h^{k}\|u\|_{k-1}\3bar \sigma\3bar_{w}+h^{s+1}(\|u\|_{s+1}+h\delta_{k, 2}\|u\|_{2})\3bar \sigma\3bar_{w}+h^{s+1}\|u\|_{s+1}\3bar \sigma\3bar_{w}\\&+\gamma_1^{-1/2}h^{s+1}\|u\|_{s+1}\3bar\sigma\3bar_{w}\\
\lesssim&(h^{k}\|u\|_{k-1}+h^{s+1}(1+\gamma_1^{-1/2})\|u\|_{s+1}+h^{s+2}\delta_{k, 2}\|u\|_{2})\3bar\sigma\3bar_{w}.
\end{split}
\end{equation}
Substituting the above inequality into \eqref{Error-error-estimate-2} with $\sigma=\epsilon_{h}$ gives
\begin{equation*}
\begin{split}
\3bar\epsilon_{h}\3bar_{w}^2+\3bar\epsilon_{h}\3bar_{c}^2&=|\zeta_u(\epsilon_{h})|\\
&\lesssim(h^{k}\|u\|_{k-1}+h^{s+1}(1+\gamma_1^{-1/2})\|u\|_{s+1}+h^{s+2}\delta_{k, 2}\|u\|_{2})\3bar\epsilon_{h}\3bar_{w},
\end{split}
\end{equation*}
which implies
\begin{eqnarray}\label{epesti}
\3bar\epsilon_{h}\3bar_{w}&\lesssim&h^{k}\|u\|_{k-1}+h^{s+1}(1+\gamma_1^{-1/2})\|u\|_{s+1}+h^{s+2}\delta_{k, 2}\|u\|_{2}.
\label{Error-error-estimate-3}
\end{eqnarray}

Moreover, we have from the error equation \eqref{error-equation-1} that
\begin{equation}\label{Error-error-estimate-4}
b(\sigma,e_{h})=\zeta_{u}(\sigma)-\SS(\epsilon_{h},\sigma),\qquad\forall\sigma\in V_{h,k}^0.
\end{equation}

It follows from Lemma \ref{inf-sup condition} that there exists a weak function $\tilde{\sigma}\in V_{h,k}^0$ and $\beta>0$ satisfying
\begin{eqnarray} \label{Error-error-estimate-6}
&|b(\tilde{\sigma},e_{h})|\geq \frac{1}{2}\|e_{h}\|^{2},\\ &\3bar\tilde{\sigma}\3bar_{w}\leq\beta\|e_{h}\|.\label{Error-error-estimate-7}
\end{eqnarray}
Using the Cauchy-Schwarz inequality gives
\begin{eqnarray}
|s(\epsilon_{h},\sigma)|\leq\3bar\epsilon_{h}\3bar_{w}\3bar\sigma\3bar_{w}.\label{Error-error-estimate-8}
\end{eqnarray}
From \eqref{stabilizer}-\eqref{additionc}, \eqref{stabilizer definition-1}-\eqref{stabilizer definition-2} and the Cauchy-Schwarz inequality, we have
\begin{equation}\label{Error-error-estimate-21}
\begin{split}
|c(\epsilon_{h},\sigma)|&\leq\3bar\epsilon_{h}\3bar_{c}\3bar\sigma\3bar_{c}\\
&\lesssim\3bar\epsilon_{h}\3bar_{w}\3bar\sigma\3bar_{w}.
\end{split}
\end{equation}
Taking $\sigma=\tilde{\sigma}$ in \eqref{Error-error-estimate-4} and using \eqref{Error-error-estimate-5-5}-\eqref{epesti}, \eqref{Error-error-estimate-6}-\eqref{Error-error-estimate-21} together yield
\begin{equation*}\label{Error-error-estimate-9}
\begin{split}
\frac{1}{2}\|e_{h}\|^{2}&\leq|\zeta_{u}(\tilde{\sigma})|+|\SS(\epsilon_{h},\tilde{\sigma})|\\
&\lesssim \3bar\tilde{\sigma}\3bar_{w}(h^{k}\|u\|_{k-1}+h^{s+1}(1+\gamma_1^{-1/2})\|u\|_{s+1}+h^{s+2}\delta_{k, 2}\|u\|_{2})\\
&\lesssim\|e_h\|(h^{k}\|u\|_{k-1}+h^{s+1}(1+\gamma_1^{-1/2})\|u\|_{s+1}+h^{s+2}\delta_{k, 2}\|u\|_{2}),
\end{split}
\end{equation*}
which leads to
\begin{equation}\label{Error-error-estimate-10}
\begin{split}
\|e_{h}\|\lesssim  h^{k}\|u\|_{k-1}+h^{s+1}(1+\gamma_1^{-1/2})\|u\|_{s+1}+h^{s+2}\delta_{k, 2}\|u\|_{2}.
\end{split}
\end{equation}

Combining the estimates \eqref{Error-error-estimate-3} and \eqref{Error-error-estimate-10} completes the proof of the theorem.
\end{proof}

\begin{corollary}\label{Corollary:error-estimate-88}
Under the assumptions of Theorem \ref{THM:energy-estimation}, the following error estimate holds true
\begin{equation*}\label{Corollary:error-estimate-1}
\begin{split}
\|u_h-u\|\lesssim h^{k}\|u\|_{k-1}+h^{s+1}(1+\gamma_1^{-1/2})\|u\|_{s+1}+h^{s+2}\delta_{k, 2}\|u\|_{2}.
\end{split}
\end{equation*}
\end{corollary}
\begin{proof} Using the triangle inequality, Lemma \ref{Section:error estimate for projections-1}, and \eqref{Error-error-estimate-10} completes the proof of the corollary.
\end{proof}
\begin{theorem}\label{Corollary:error-estimate-3}
Let $s=k-1$ and $k\geq 1$.  Assume that the diffusion tensor $a(x)$ and the drift vector $\pmb{\mu}$ are piecewise constants with respect to the finite element partition ${\cal T}_h$ which is shape regular. Let $u$ and $(u_h;\rho_h)\in W_{h, s}\times V_{h, k}^0$ be the exact solution of the model problem \eqref{model-problem} and the numerical solution of the S-PDWG scheme \eqref{PDWG-scheme1}-\eqref{PDWG-scheme2} with the stabilization parameter $\gamma_1\geq0$, respectively. Assume that the exact solution of the model equation \eqref{model-problem}  is sufficiently regular satisfying $u\in\prod_{T\in {\cal T}_h}H^{k}(T)\cap H^2(T)$. The following error estimate holds true
\begin{equation*}\label{Corollary:error-estimate-8}
\begin{split}
\|u-u_h\|\lesssim h^{k}(\|u\|_{k}+ h\delta_{k,2}\|u\|_2).
\end{split}
\end{equation*}
\end{theorem}
\begin{proof}
Note that the last term $({\L}\sigma_{0},{\S}^{(s)}u-u)_{T}=0$ in the remainder \eqref{error equation-3} when $s=k-1$ and the diffusion tensor $a(x)$ and the drift vector $\pmb{\mu}$ are piecewise constants.
 Following the proof of Theorem \ref{THM:energy-estimation} and Corollary \ref{Corollary:error-estimate-88} completes the proof of the theorem without any difficulty.
\end{proof}

\section{Numerical Experiments}\label{Section:NE}
\subsection{Implementation of Tangential Component ${\dt}v_{b} \pmb{\tau}$ of $\bv_g$}\label{Section:Implementation  PD-WG}

We shall first discuss the implementation of the tangential component ${\dt}v_{b} \pmb{\tau}$ by taking $k=2$ on any polygonal element as an example. It can be easily generalized to $k\geq 1$ in two and three dimensions without any difficulty.

Let $T\in{\cal T}_{h}$ be a polygonal element. Denote by $|e|$ the length of the edge $e\subset \pa T$ with the start point $A_i(x_i, y_i)$ and the end point $A_j(x_j, y_j)$. The basis functions for the space $P_1(e)$ are the linear functions $\chi_i(i=1,2)$ given by
$$\chi_1=1~\text{on}~ A_i,~\chi_1=0~\text{on}~A_j;~~\text{and}~~\chi_2=1~\text{on}~ A_j,~\chi_2=0~\text{on}~A_i.$$
 The basis functions for the space $P_2(e)$ are the quadratic functions $\varphi_{bi} (i=1,2,3)$ given by
 \begin{equation*}
\begin{split}
\varphi_{b1}=\chi_1(2\chi_1-1),~~\varphi_{b2}= 4\chi_1\chi_2,~~\varphi_{b3}= \chi_2(2\chi_2-1).
\end{split}
\end{equation*}
For any function $v_b\in P_2(e)$, we have $v_b=\sum_{i=1}^3v_{bi}\varphi_{bi}$ with $v_{bi}(i=1,2,3)$ being the coefficients to be determined. Thus, we have

\begin{equation*}
\begin{split}
{\dt}v_{b} \pmb{\tau}= &{\dt}(v_{b1}\varphi_{b1}+v_{b2}\varphi_{b2}+v_{b3}\varphi_{b3})\pmb{\tau}\\
=&\left((4\chi_1-1){\dt}\chi_1v_{b1}+4(\chi_1{\dt}\chi_2+\chi_2{\dt}\chi_1)v_{b2}+(4\chi_2-1){\dt}\chi_2v_{b3}\right)\pmb{\tau}\\
=&\frac{-1}{|e|}\left((4\chi_1-1)v_{b1}+4(\chi_2-\chi_1)v_{b2}-(4\chi_2-1)v_{b3}\right)\pmb{\tau},
\end{split}
\end{equation*}
where we used  ${\dt}\chi_1=\frac{1}{|e|}(\chi_1(A_j)-\chi_1(A_i))=\frac{-1}{|e|}$, ${\dt}\chi_2=\frac{1}{|e|}(\chi_1(A_j)-\chi_1(A_i))=\frac{1}{|e|}$.

\subsection{Numerical Tests} We shall demonstrate some numerical examples to verify the theoretical results established in Section \ref{Section:error estimate}.

The convex domains are given by two square domains $\O_1=(0,1)^2$ and $\O_2=(-1,1)^2$. The non-convex domains are given by the L-shaped domain $\O_3=(0,1)^2\backslash (0.5,1)^2$ and the cracked domain $\O_4=\{|x|+|y|<1\}\backslash (0,1)*0$. The uniform triangular, rectangular and square partitions are employed in the numerical tests. The triangular partition starts from an initial triangulation of the domain and the meshes are successively refined by connecting the midpoints of the edges of each triangle. The rectangular partition is obtained from an initial $3 \times 2$ rectangular mesh which is  successively refined by connecting the midpoints of the parallel edges of each rectangle. The square partition is obtained from an initial $2\times 2$ square mesh and the next level is obtained by connecting the midpoints of the parallel edges of each square.

Let $k\geq1$. Recall that the finite element spaces for the primal variable $u_h$ and its dual variable $\rho_h$ are given as follows
\begin{equation*}
\begin{split}
W_{h, s}=&\{w:w|_{T}\in P_{s}(T), s=k-1~\text{or}~k-2,\; \forall T\in {\cal T}_{h}\},
\\
V_{h, k}=&\{\{v_{0},v_{b},\pmb{v_{g}}=v_{n}\bn_{e}+{\dt}v_{b}\pmb{\tau}\}:v_{0}\in P_{k}(T),v_{b}\in P_{k}(\pa T), v_{n}\in P_{k-1}(\pa T),  \forall T\in {\cal T}_{h}\}.
\end{split}
\end{equation*}
The corresponding element is called the ``Simplified $C^{-1}-P_{k}(T)/P_{k}(\pa T)/P_{k-1}(\pa T)/P_{s}(T)$ element''.  The error functions are measured in the following norms; i.e.,
$$
 \|\epsilon_0\|_0= \left(\sum_{T\in {\cal T}_h} \int_T |\epsilon_0|^2 dT\right)^{\frac{1}{2}}, \qquad \|\epsilon_b\|_0=\left(\sum_{T\in {\cal T}_h}h_T \int_{\partial T} |\epsilon_b|^2 ds\right)^{\frac{1}{2}},
$$
$$
\3bar\epsilon_n\3bar_1=\left(\sum_{T\in{\cal T}_{h}}h_T\int_{\pa T}|\epsilon_n|^2ds\right)^{1/2},  \qquad \| e_h\|_0= \left(\sum_{T\in{\cal T}_{h}}\int_T|e_h|^2dT\right)^{1/2}.
$$

First of all, we shall compare the degrees of freedom(dofs) of the S-PDWG method \eqref{PDWG-scheme1}-\eqref{PDWG-scheme2} with the PDWG scheme proposed in \cite{WW-fp-2018}, where the corresponding element is called ``General $C^{-1}$-type element''.
For a finite element partition ${\cal T}_{h}$, denote by $NT$ the number of elements and $NE$ the number of edges or faces. Note that the space $W_{h, s}$ is the same for both the S-PDWG scheme and the PDWG scheme  \cite{WW-fp-2018}. As to the space $V_{h, k}$,  it is obvious to see from Tables \ref{dofff-1}-\ref{dofff-2} that the  ``Simplified $C^{-1}-type$ element'' proposed in our paper has significantly fewer dofs than the ``General $C^{-1}-type$ element'' proposed in  \cite{WW-fp-2018} on any polygonal  and polyhedral meshes. 

\begin{table}[htbp!]\tiny\centering\scriptsize
\caption{Comparison of dofs for $k\geq1$ on any polygonal meshes.}\label{dofff-1}
{
\setlength{\extrarowheight}{1.5pt}
\begin{center}
\begin{tabular}{|l|l|}
\hline  &$V_{h, k}$   \\
\hline
 General $C^{-1}-type$      &$\frac{1}{2}(k+1)(k+2)NT+(3k+1)NE$   \\ \hline
Simplified $C^{-1}-type$             &$\frac{1}{2}(k+1)(k+2)NT+(2k+1)NE$     \\ \hline
\end{tabular}
\end{center}
}
\end{table}

\begin{table}[htbp!]\tiny\centering\scriptsize
\caption{Comparison of dofs for $k\geq1$ on any polyhedral meshes.}\label{dofff-2}
{
\setlength{\extrarowheight}{1.5pt}
\begin{center}
\begin{tabular}{|l|l|}
\hline  &$V_{h,k}$   \\
\hline
 General $C^{-1}-type$      &$\frac{1}{6}(k+1)(k+2)(k+3)NT+\frac{1}{2}(k+1)(3k+2)NE$    \\ \hline
Simplified $C^{-1}-type$      &$\frac{1}{6}(k+1)(k+2)(k+3)NT+(k+1)^2NE$   \\ \hline
\end{tabular}
\end{center}
}
\end{table}

\subsubsection{Numerical experiments with continuous diffusion tensor}
Table \ref{Example1:1:SquareTRI:test1-1} illustrates the performance of simplified $C^{-1}$-$P_2(T)/P_2(\pa T)/P_1(\pa T)/P_s(T)$ element on the uniform triangular partition of the unit square domain $\Omega_1=(0, 1)^2$ when $s=0$ and $s=1$ are employed, respectively. The exact solution is $u=\sin(x)\cos(y)$; the diffusion tensor $a=\{a_{ij}\}$ is $a_{11}=3,\  a_{12}=a_{21}=1,\ a_{22}=2$; the drift vector is $\pmb{\mu}=[1,1]'$; and the stabilization parameters are $\gamma_1=\gamma_2=\gamma_3=1$. We observe from Table  \ref{Example1:1:SquareTRI:test1-1} that the convergence rate for $e_h$ in the $L^2$ norm is of the expected optimal order ${\cal O}(h^{2})$ for $s=1$ and of an order higher than the expected optimal order ${\cal O}(h)$ for $s=0$, respectively.

\begin{table}[htbp]\centering\scriptsize
\caption{Convergence rates for simplified $C^{-1}$-$P_2(T)/P_2(\pa T)/P_1(\pa T)/P_s(T)$ element with exact solution $u=\sin(x)\cos(y)$ on $\O_1$; uniform triangular partition;
 the diffusion tensor $a_{11}=3$, $a_{12}=a_{21}=1$, and $a_{22}=2$; the drift vector $\pmb{\mu}=[1,1]'$; the parameters $\gamma_1=\gamma_2=\gamma_3=1$.}\label{Example1:1:SquareTRI:test1-1}
{
\setlength{\extrarowheight}{1.5pt}
\begin{center}
\begin{tabular}{|l|l|l|l|l|l|l|l|l|l|}
\hline  &$1/h$ &$\| \epsilon_0\|_0$& $Rate$ & $\| \epsilon_b\|_0$  &$Rate$ & $\3bar\epsilon_n\3bar_1$& $Rate$ & $\|e_h\|_0$& $Rate$ \\
\hline
   &$2$      &6.853e-02 & &7.679e-02   &  &6.718e-01     &   &3.815e-02  &\\
s=0&$4$      &1.446e-02 &2.25  &1.877e-02   &2.03   &1.642e-01     &2.03   &1.474e-02  &1.37\\
  &$8$       &3.626e-03 &2.00  &5.014e-03   &1.90   &4.075e-02     &2.01   &4.697e-03  &1.65\\
  &$16$      &9.230e-04 &1.97  &1.298e-03   &1.95   &1.015e-02     &2.01   &1.712e-03  &1.46 \\
  &$32$      &2.339e-04 &1.98  &3.303e-04   &1.97   &2.531e-03     &2.00   &7.256e-04  &1.24     \\
 \hline &$1/h$ &$\| \epsilon_0\|_0$& $Rate$ & $\| \epsilon_b\|_0$  &$Rate$ &$\3bar\epsilon_n\3bar_1$& $Rate$ & $\|e_h\|_0$& $Rate$ \\
\hline
&$2$     &8.536e-03 & &5.855e-03   &   &2.134e-02     &   &4.073e-02 &\\
s=1&$4$  &5.756e-04 &3.89  &4.479e-04   &3.71   &2.200e-03     &3.28   &9.947e-03 &2.03\\
&$8$     &3.718e-05 &3.95  &2.979e-05   &3.91   &2.427e-04     &3.18   &2.450e-03 &2.02\\
&$16$    &2.372e-06 &3.97  &1.915e-06   &3.96   &2.852e-05     &3.09   &6.083e-04 &2.01 \\
&$32$    &1.500e-07 &3.98  &1.213e-07   &3.98   &3.463e-06     &3.04   &1.517e-04 &2.00\\               \hline
\end{tabular}
\end{center}
}
\end{table}

Table \ref{Example2:0:SquareREC:test1-1} demonstrates the performance of simplified $C^{-1}$-$P_k(T)/P_k(\pa T)/P_{k-1}(\pa T)/P_0(T)$ element for the cases of $k=1$ and $k=2$ when the uniform rectangular partition is employed on the unit square domain $\Omega_1=(0,1)^2$. The exact solution is given by $u=\cos(\pi x)\cos(\pi y)$; the diffusion tensor $a=\{a_{ij}\}$ is $a_{11}=3,\  a_{12}=a_{21}=1,\ a_{22}=2$; and the drift vector is $\pmb{\mu}=[1,1]'$. The  parameters are $\gamma_2=\gamma_3=1$. It can been seen from Table \ref{Example2:0:SquareREC:test1-1} that the convergence rate for $e_h$ in the $L^2$ norm arrives at a superconvergence order which outperforms the optimal order ${\cal O}(h)$ for the cases of $(k, \gamma_1)=(1,0)$ and  $(k, \gamma_1)=(2,1)$, respectively.
\begin{table}[htbp]\centering\scriptsize
\caption{Convergence rates for simplified $C^{-1}$-$P_k(T)/P_k(\pa T)/P_{k-1}(\pa T)/P_0(T)$ element  with exact solution $u=\cos(\pi x)\cos(\pi y)$ on $\O_1$; uniform rectangular partition;
 the diffusion tensor $a_{11}=3$, $a_{12}=a_{21}=1$, and $a_{22}=2$; the drift vector $\pmb{\mu}=[1,1]'$; the parameters $\gamma_2=\gamma_3=1$.}\label{Example2:0:SquareREC:test1-1}
{
\setlength{\extrarowheight}{1.5pt}
\begin{center}
\begin{tabular}{|l|l|l|l|l|l|l|l|l|l|}
\hline &$1/h$ &$\| \epsilon_0\|_0$& $Rate$ & $\| \epsilon_b\|_0$  &$Rate$ &    $\3bar\epsilon_n\3bar_1$& $Rate$ & $\|e_h\|_0$& $Rate$ \\
\hline
&$4$             &1.555e-02 &  &3.066e-02   &   &1.680e-01     &   &8.518e-02  &\\
k=1&$8$          &4.631e-03 &1.75  &9.582e-03   &1.68   &4.951e-02     &1.76   &4.320e-02  &0.98\\
$\gamma_1=0$&$16$&1.293e-03 &1.84  &2.690e-03   &1.83   &1.356e-02     &1.87   &1.811e-02  &1.25\\
&$32$            &3.414e-04 &1.92  &7.107e-04   &1.92   &3.538e-03     &1.94   &6.756e-03  &1.42   \\
&$64$            &8.767e-05 &1.96  &1.825e-04   &1.96   &9.023e-04     &1.97   &2.447e-03  &1.47\\
\hline &$1/h$ &$\| \epsilon_0\|_0$& $Rate$ & $\| \epsilon_b\|_0$  &$Rate$ &    $\3bar\epsilon_n\3bar_1$& $Rate$ & $\|e_h\|_0$& $Rate$ \\
\hline
&$4$            &1.886e-02 &  &4.242e-02   &   &2.443e-01     &   &4.323e-02  &\\
k=2&$8$         &5.096e-03 &1.89  &1.263e-02   &1.75   &6.856e-02     &1.83   &1.033e-02  &2.07\\
$\gamma_1=1$&$16$&1.347e-03 &1.92 &3.411e-03   &1.89   &1.775e-02     &1.95   &3.258e-03  &1.67\\
&$32$           &3.474e-04 &1.96  &8.842e-04   &1.95   &4.501e-03     &1.98   &1.040e-03  &1.65   \\
&$64$           &8.826e-05 &1.98  &2.249e-04   &1.97   &1.133e-03     &1.99   &3.401e-04  &1.61\\               \hline
\end{tabular}
\end{center}
}
\end{table}

Table \ref{Example3:0:square-shaped-TRI:test2-2-1} shows the numerical results for the test problem \eqref{model-problem} when the exact solution is given by $u=\sin(x)\cos(y)$ on the uniform triangular partition of the  domain $\O_1=(0,1)^2$. The diffusion tensor $a=\{a_{ij}\}$ is $a_{11}=1+x^2$, $a_{12}=a_{21}=0.25xy$, and $a_{22}=1+y^2$; and the drift vector is $\pmb{\mu}=[x,y]'$. The parameters are given by $\gamma_1=\gamma_2=\gamma_3=1$. Our numerical results indicate the convergence rate for $e_h$ in the $L^2$ norm is of an expected optimal order ${\cal O}(h^2)$ for simplified $C^{-1}$-$P_2(T)/P_2(\pa T)/P_{1}(\pa T)/P_{1}(T)$ element and of an order higher than  the expected optimal order ${\cal O}(h)$ for simplified $C^{-1}$-$P_1(T)/P_1(\pa T)/P_{0}(\pa T)/P_{0}(T)$ element, respectively.

\begin{table}[htbp]\centering\scriptsize
\caption{Convergence rates for simplified $C^{-1}$-$P_k(T)/P_k(\pa T)/P_{k-1}(\pa T)/P_{k-1}(T)$ element with exact solution $u=\sin(x)\cos(y)$ on $\O_1$; uniform triangular partition; the diffusion tensor $a_{11}=1+x^2$, $a_{12}=a_{21}=0.25xy$, and $a_{22}=1+y^2$; the drift vector $\pmb{\mu}=[x,y]'$; the parameters $\gamma_1=\gamma_2=\gamma_3=1$.}\label{Example3:0:square-shaped-TRI:test2-2-1}
{
\setlength{\extrarowheight}{1.5pt}
\begin{center}
\begin{tabular}{|l|l|l|l|l|l|l|l|l|l|}
\hline &$1/h$ &$\| \epsilon_0\|_0$& $Rate$ & $\| \epsilon_b\|_0$  &$Rate$ &    $\3bar\epsilon_n\3bar_1$& $Rate$ & $\|e_h\|_0$& $Rate$ \\
\hline
&$2$           &2.940e-02 &  &4.473e-02   &   &1.606e-01     &   &3.070e-02 &\\
&$4$           &6.760e-03 &2.12  &1.006e-02   &2.15   &5.995e-02     &1.42   &1.724e-02 &0.83\\
$k=1$&$8$      &1.368e-03 &2.31  &2.077e-03   &2.28   &1.955e-02     &1.62   &7.552e-03 &1.19\\
&$16$          &2.864e-04 &2.26  &4.522e-04   &2.20   &5.522e-03     &1.82   &3.153e-03 &1.26  \\
&$32$          &6.645e-05 &2.11  &1.074e-04   &2.07   &1.452e-03     &1.93   &1.433e-03 &1.14\\
\hline &$1/h$ &$\| \epsilon_0\|_0$& $Rate$ & $\| \epsilon_b\|_0$  &$Rate$ &    $\3bar\epsilon_n\3bar_1$& $Rate$ & $\|e_h\|_0$& $Rate$ \\
\hline
&$2$           &6.523e-03 &  &7.606e-03   &  &3.726e-02     &  &8.630e-02 &\\
$k=2$&$4$      &5.269e-04 &3.63  &6.467e-04   &3.56   &4.850e-03     &2.94   &2.252e-02 &1.94\\
&$8$           &3.540e-05 &3.90  &4.381e-05   &3.88   &5.557e-04     &3.13   &5.603e-03 &2.01\\
&$16$          &2.272e-06 &3.96  &2.817e-06   &3.96   &6.527e-05     &3.09   &1.397e-03 &2.00  \\
&$32$          &1.438e-07 &3.98  &1.784e-07   &3.98   &7.956e-06     &3.04   &3.488e-04 &2.00\\               \hline
\end{tabular}
\end{center}
}
\end{table}

Tables \ref{Example4:0:L-shaped-tri:test2-2-1}-\ref{Example5:0:Crack-tri:test2-2-0} illustrate the numerical performance of simplified $C^{-1}$-$P_2(T)/P_2(\pa T)/P_{1}(\pa T)/P_s(T)$ element when the exact solution is given by $u=\sin(x)\cos(y)$ on the uniform triangular partition of the non-convex L-shaped domain $\O_3$ and cracked domain $\O_4$. The diffusion tensor is $a_{11}=1+x^2$, $a_{12}=a_{21}=0.25xy$, and $a_{22}=1+y^2$; and the drift vector is $\pmb{\mu}=[x,y]'$. The parameters are $\gamma_2=\gamma_3=1$.  It can be seen from Table \ref{Example4:0:L-shaped-tri:test2-2-1} that the convergence rate for $e_h$ in the $L^2$ norm seems to be of an order ${\cal O}(h)$ for $(s, \gamma_1)=(0, 0.1)$;  and of an order ${\cal O}(h^2)$ when $(s, \gamma_1)=(1, 10000)$ on the L-shaped domain $\Omega_3$. We observe from Table \ref{Example5:0:Crack-tri:test2-2-0} that the convergence rate for $e_h$ in the $L^2$ norm seems to be of an order ${\cal O}(h^2)$ for the simplified $C^{-1}$-$P_2(T)/P_2(\pa T)/P_{1}(\pa T)/P_1(T)$ element on the cracked domain $\Omega_4$ for the cases of $\gamma_1=0$ and $\gamma_1=1$. It should be pointed out that the convergence theory has not been developed on non-convex domains in this paper. However, we can still observe the convergence order from the numerical results when the non-convex domains are employed.

\begin{table}[htbp]\centering\scriptsize
\caption{Convergence rates for simplified $C^{-1}$-$P_2(T)/P_2(\pa T)/P_1(\pa T)/P_s(T)$ element with the exact solution $u=\sin(x)\cos(y)$ on the L-shaped domain $\O_3$; uniform triangular partition; the diffusion tensor $a_{11}=1+x^2$, $a_{12}=a_{21}=0.25xy$, and $a_{22}=1+y^2$; the drift vector $\pmb{\mu}=[x,y]'$; the parameters  $\gamma_2=\gamma_3=1$.}\label{Example4:0:L-shaped-tri:test2-2-1}
{
\setlength{\extrarowheight}{1.5pt}
\begin{center}
\begin{tabular}{|l|l|l|l|l|l|l|l|l|l|}
\hline &$1/h$ &$\| \epsilon_0\|_0$& $Rate$ & $\| \epsilon_b\|_0$  &$Rate$ &    $\3bar\epsilon_n\3bar_1$& $Rate$ & $\|e_h\|_0$& $Rate$ \\
\hline
&$2$           &1.108e-01       &      &1.374e-01   &       &9.854e-01     &       &1.438e-01 &\\
$s=0$&$4$      &2.147e-02       &2.37  &2.865e-02   &2.26   &2.499e-01     &1.98   &6.249e-02 &1.20\\
$\gamma_1=0.1$&$8$&4.614e-03    &2.22  &6.395e-03   &2.16   &6.337e-02     &1.98   &2.793e-02 &1.16\\
&$16$           &1.087e-03      &2.09  &1.529e-03   &2.06   &1.599e-02     &1.99   &1.315e-02 &1.09  \\
\hline &$1/h$ &$\| \epsilon_0\|_0$& $Rate$ & $\| \epsilon_b\|_0$  &$Rate$ &    $\3bar\epsilon_n\3bar_1$& $Rate$ & $\|e_h\|_0$& $Rate$ \\
\hline
&$2$               &1.667e-02 &    &1.776e-02   &  &3.495e-02     &  &1.880e-01 &\\
$s=1$&$4$          &1.258e-03 &3.73     &1.403e-03   &3.66   &4.163e-03     &3.07   &4.573e-02 &2.04\\
$\gamma_1=10000$&$8$ &8.516e-05 &3.88     &9.683e-05   &3.86   &4.741e-04     &3.13   &1.126e-02 &2.02\\
&$16$              &5.510e-06 &3.95      &6.305e-06    &3.94   &5.387e-05     &3.14  &  2.817e-03   &2.00 \\ \hline
\end{tabular}
\end{center}
}
\end{table}
\begin{table}[htbp]\centering\scriptsize
\caption{Convergence rates for simplified $C^{-1}$-$P_2(T)/P_2(\pa T)/P_1(\pa T)/P_1(T)$ element with the exact solution $u=\sin(x)\cos(y)$ on the crack domain $\O_4$; uniform triangular partition; the diffusion tensor $a_{11}=1+x^2$, $a_{12}=a_{21}=0.25xy$, and $a_{22}=1+y^2$; the drift vector $\pmb{\mu}=[x,y]'$; the parameters $\gamma_2=\gamma_3=1$.}\label{Example5:0:Crack-tri:test2-2-0}
{
\setlength{\extrarowheight}{1.5pt}
\begin{center}
\begin{tabular}{|l|l|l|l|l|l|l|l|l|l|}
\hline &$2/h$ &$\| \epsilon_0\|_0$& $Rate$ & $\| \epsilon_b\|_0$  &$Rate$ &    $\3bar\epsilon_n\3bar_1$& $Rate$ & $\|e_h\|_0$& $Rate$ \\
\hline
&$2$            &8.930e-03 & &7.834e-03   &   &5.152e-02     &  &9.903e-02 &\\
&$4$            &6.504e-04 &3.78  &6.676e-04   &3.55   &5.055e-03     &3.35   &2.537e-02 &1.96\\
$\gamma_1=0$&$8$&4.249e-05 &3.94  &4.579e-05   &3.87   &5.187e-04     &3.28   &6.285e-03 &2.01\\
&$16$           &2.704e-06 &3.97  &2.975e-06   &3.94   &5.865e-05     &3.14   &1.556e-03 &2.01  \\
 \hline &$2/h$ &$\| \epsilon_0\|_0$& $Rate$ & $\| \epsilon_b\|_0$  &$Rate$ &    $\3bar\epsilon_n\3bar_1$& $Rate$ & $\|e_h\|_0$& $Rate$ \\
\hline
&$2$            &8.930e-03 & &7.869e-03   &   &4.767e-02     &   &1.036e-01 &\\
$\gamma_1=1$&$4$&6.501e-04 &3.78  &6.717e-04   &3.55   &5.050e-03     &3.24   &2.590e-02 &2.00\\
&$8$            &4.255e-05 &3.93  &4.618e-05   &3.86   &5.213e-04     &3.28   &6.371e-03 &2.02\\
&$16$           &2.708e-06 &3.97  &3.000e-06   &3.94   &5.886e-05     &3.15   &1.573e-03 &2.02  \\ \hline
\end{tabular}
\end{center}
}
\end{table}

\subsubsection{Numerical experiments with discontinuous diffusion tensor}
Table \ref{Example28:square-tri:test3-2-0-1} illustrates some numerical results for simplified $C^{-1}$-$P_2(T)/P_2(\pa T)/P_1(\pa T)/P_1(T)$ element on the uniform triangular partition of the domain $\O_1$. The configuration of this test example is set as follows: $(1)$ the exact solution is given by $u=2\sin(2x)\cos(3y)$ when $y<1-x$ and $u=\sin(2x)\cos(3y)$ elsewhere; $(2)$ the diffusion tensor is piece-wisely defined in the sense that $a=I$ for $y<1-x$ and $a=2I$ otherwise, where $I$ is an identity matrix; (3) the drift vector is $\pmb{\mu}=[0,0]'$; and (4)  the  parameters are chosen as $\gamma_1=\gamma_2=\gamma_3=1$. We observe from Table \ref{Example28:square-tri:test3-2-0-1} that the convergence rate for $e_h$ in $L^2$ norm is of an order ${\cal O}(h^2)$, which supports the theory developed in Theorem \ref{Corollary:error-estimate-3}. 
\begin{table}[htbp]\centering\scriptsize
\caption{Convergence rates for simplified $C^{-1}$-$P_2(T)/P_2(\pa T)/P_1(\pa T)/P_1(T)$ element on $\O_1$; uniform triangular partition; $a=I$,  $u=2\sin(2x)\cos(3y)$ for $y<1-x$ and $a=2I$, $u=\sin(2x)\cos(3y)$ elsewhere; the drift vector $\pmb{\mu}=[0,0]'$; the parameter $\gamma_1=\gamma_2=\gamma_3=1$.} \label{Example28:square-tri:test3-2-0-1}
{
\setlength{\extrarowheight}{1.5pt}
\begin{center}
\begin{tabular}{|l|l|l|l|l|l|l|l|l|}
\hline $1/h$ &$\| \epsilon_0\|_0$& $Rate$ & $\| \epsilon_b\|_0$  &$Rate$ &    $\3bar\epsilon_n\3bar_1$& $Rate$ & $\|e_h\|_0$& $Rate$ \\
\hline
$2$     &3.958e-02 &   &2.697e-02   &    &1.263e-01     &   &1.028e-01  &\\
$4$     &2.909e-03 &3.77   &2.217e-03   &3.60    &1.404e-02     &3.17   &3.029e-02  &1.76\\
$8$     &1.931e-04 &3.91   &1.529e-04   &3.86    &1.573e-03     &3.16   &8.094e-03  &1.90\\
$16$    &1.240e-05 &3.96   &9.996e-06   &3.94    &1.890e-04     &3.06   &2.068e-03  &1.97 \\
$32$    &7.849e-07 &3.98   &6.374e-07   &3.97    &2.331e-05     &3.02   &5.195e-04  &1.99\\               \hline
\end{tabular}
\end{center}
}
\end{table}

Table \ref{Example7:0:square-REC:test3-2-0-1} shows the numerical results for simplified $C^{-1}$-$P_2(T)/P_2(\pa T)/P_{1}(\pa T)/P_0(T)$ element on the uniform rectangular partition of the domain $\O_2=(-1,1)^2$. The exact solution is given by $u=\alpha^{i}\sin(\pi x)\sin(\pi y)$, and the diffusion tensor $a=\{a_{ij}\}$ is given by $a_{11}=a_{11}^{i}$, $a_{12}=a_{21}=0$, $a_{22}=a_{22}^{i}$, where $a_{11}^{i}$, $a_{22}^{i}$ and $\alpha^{i} (i=1, \cdots, 4)$ are specified in Table \ref{Example7:coefficient parameter} with the superscript $i$ corresponding to the values in the $i-$th quadrant. The drift vector is $\pmb{\mu}=[0,0]'$.
The parameters in the S-PDWG scheme \eqref{PDWG-scheme1}-\eqref{PDWG-scheme2} are given by $\gamma_1=\gamma_2=\gamma_3=1$. The numerical results indicate that the convergence rate for $e_h$ in the $L^2$ norm arrives at an order of ${\cal O}(h^{1.8})$, which outperforms the theoretical prediction ${\cal O}(h)$.

\begin{table}[htbp]\centering\scriptsize
 {\caption{The diffusion tensor $a=\{a_{ij}\}$ and the exact solution $u=\alpha^{i}\sin(\pi x)\sin(\pi y)$.} \label{Example7:coefficient parameter}
{
\setlength{\extrarowheight}{1.5pt}
\begin{center}
\begin{tabular}{|l|l|}
\hline
$a_{11}^{2}=0.1$    &$a_{11}^{1}=1000$\\
$a_{22}^{2}=0.01$   &$a_{22}^{1}=100$\\
$\alpha^{2}=100$        &$\alpha^{1}=0.01$    \\ \hline
$a_{11}^{3}=100$    &$a_{11}^{4}=1$\\
$a_{22}^{3}=10$     &$a_{22}^{4}=0.1$\\
$\alpha^{3}=0.1$        &$\alpha^{4}=10$    \\  \hline
\end{tabular}
\end{center}
}
}\end{table}
\begin{table}[htbp]\centering\scriptsize
\caption{Convergence rates for simplified $C^{-1}$-$P_2(T)/P_2(\pa T)/P_{1}(\pa T)/P_0(T)$ element on $\O_2$; the exact solution $u=\alpha^{i}\sin(\pi x)\sin(\pi y)$;  the diffusion tensor $a_{11}=\alpha_{i}^{x}$, $a_{22}=\alpha_{i}^{y}$, $a_{12}=a_{21}=0$; the drift vector $\pmb{\mu}=[0,0]'$; uniform rectangular partition; the parameters $\gamma_1=\gamma_2=\gamma_3=1$.}\label{Example7:0:square-REC:test3-2-0-1}
{
\setlength{\extrarowheight}{1.5pt}
\begin{center}
\begin{tabular}{|l|l|l|l|l|l|l|l|l|}
\hline $2/h$ &$\| \epsilon_0\|_0$& $Rate$ & $\| \epsilon_b\|_0$  &$Rate$ &    $\3bar\epsilon_n\3bar_1$& $Rate$ & $\|e_h\|_0$& $Rate$ \\
\hline
$2$     &1.301e-00 &     &2.750e-00   &     &5.894e-00        &     &2.480e+01  & \\
$4$     &3.413e-01 &1.93     &8.332e-01   &1.72     &1.389e-00        &2.08     &6.464e-00  &1.94 \\
$8$     &8.406e-02 &2.02     &2.120e-01   &1.97     &3.139e-01        &2.15     &1.830e-00  &1.82\\
$16$    &2.088e-02 &2.01     &5.308e-02   &2.00     &7.476e-02        &2.07     &5.281e-01  &1.79\\
$32$    &5.209e-03 &2.00     &1.327e-02   &2.00     &1.831e-02        &2.03     &1.444e-01  &1.87 \\           \hline
\end{tabular}
\end{center}
}
\end{table}

Table \ref{Example6:0:square-shaped-TRI:test2-2-0} demonstrates the numerical performance on the uniform triangular partition of the domain $\O_2=(-1,1)^2$ for simplified $C^{-1}$-$P_2(T)/P_2(\pa T)/P_1(\pa T)/P_1(T)$ element. The exact solution is given by $u=\frac{1}{\alpha^i}\cos(x)\cos(y)$; and the diffusion tensor $a(x)=\{a_{ij}\}$ is $a_{11}=a_{11}^i$, $a_{12}=a_{21}=0$ and $a_{22}=a_{22}^i$. Here $a_{11}^{i}$, $a_{22}^{i}$ and $\alpha^{i}$ ($i=1, \cdots, 4$) are detailed in Table \ref{Example6:coefficient parameter} with the superscript $i$ corresponding to the values in the $i-$th quadrant. The drift vector $\pmb{\mu}=[0,0]'$.  The parameters are $\gamma_2=\gamma_3=1$. It can be seen from Table \ref{Example6:0:square-shaped-TRI:test2-2-0} that the convergence rate for $e_h$ in the $L^2$ norm arrives at an order of ${\cal O}(h^{2})$ for the case of $\gamma_1=1$, which is consistent with the theory. As to the case of $\gamma_1=0$,
Table \ref{Example6:0:square-shaped-TRI:test2-2-0} shows that the convergence rate for $e_h$ in the $L^2$ norm seems to arrive at an order of ${\cal O}(h^{2})$, for which no theory is available to compare with.

\begin{table}[htbp]\centering\scriptsize
{\color{black}{\caption{The diffusive tensor $a(x)=\{a_{ij}\}$ and the exact solution $u=\frac{1}{\alpha^i}\cos(x)\cos(y)$.} \label{Example6:coefficient parameter}
{
\setlength{\extrarowheight}{1.5pt}
\begin{center}
\begin{tabular}{|l|l|}
\hline
$a_{11}^{2}=2(2+x^{2})$    &$a_{11}^{1}=2+x^{2}$\\
$a_{22}^{2}=2(2-y^{2})$    &$a_{22}^{1}=2+y^{2}$\\
$\alpha^{2}=2$              &$\alpha^{1}=1$    \\ \hline
$a_{11}^{3}=3(2-x^2)$      &$a_{11}^{4}=4(2-x^2)$\\
$a_{22}^{3}=3(2-y^{2})$    &$a_{22}^{4}=4(2+y^{2})$\\
$\alpha^{3}=3$              &$\alpha^{4}=4$    \\  \hline
\end{tabular}
\end{center}
}}
}
\end{table}
\begin{table}[htbp]\centering\scriptsize
\caption{Convergence rates for simplified $C^{-1}$-$P_2(T)/P_2(\pa T)/P_1(\pa T)/P_1(T)$ element on $\O_2$; the exact solution $u=\frac{1}{\alpha^i}cos(x)cos(y)$ and the diffusion tensor $a(x)$ defined in Table \ref{Example6:coefficient parameter}; uniform triangular partition;  the drift vector $\pmb{\mu}=[0,0]'$; the parameters $\gamma_2=\gamma_3=1$.} \label{Example6:0:square-shaped-TRI:test2-2-0}
{
\setlength{\extrarowheight}{1.5pt}
\begin{center}
\begin{tabular}{|l|l|l|l|l|l|l|l|l|l|l|}
\hline &$2/h$ &$\| \epsilon_0\|_0$& $Rate$ & $\| \epsilon_b\|_0$  &$Rate$ &    $\3bar\epsilon_n\3bar_1$& $Rate$ & $\|e_h\|_0$& $Rate$ \\
\hline
&$2$           &3.001e-01 &  &1.905e-01   &    &1.052e-01        &     &4.998e-02  &  \\
$\gamma_1=0$&$4$&2.350e-02 &3.68 &1.840e-02   &3.37     &6.883e-02        &0.61     &1.342e-02  &1.90 \\
&$8$           &1.626e-03 &3.85  &1.376e-03   &3.74     &7.820e-03        &3.14     &3.514e-03  &1.93\\
&$16$          &1.044e-04 &3.96  &9.032e-05   &3.93     &8.103e-04        &3.27     &8.808e-04  &2.00 \\
&$32$          &6.595e-06 &3.98  &5.756e-06   &3.97     &9.151e-05        &3.15     &2.194e-04  &2.01 \\
\hline &$2/h$ &$\| \epsilon_0\|_0$& $Rate$ & $\| \epsilon_b\|_0$  &$Rate$ &    $\3bar\epsilon_n\3bar_1$& $Rate$ & $\|e_h\|_0$& $Rate$ \\
\hline
&$2$            &3.067e-01 & &2.224e-01   &    &1.461e-01        &    &7.144e-02  &  \\
$\gamma_1=1$&$4$&2.355e-02 &3.70 &1.957e-02   &3.51     &6.480e-02        &1.17     &1.880e-02  &1.93 \\
&$8$            &1.650e-03 &3.83 &1.460e-03   &3.74     &7.606e-03        &3.09     &4.782e-03  &1.98\\
&$16$           &1.068e-04 &3.95 &9.636e-05   &3.92     &8.249e-04        &3.20     &1.198e-03  &2.00 \\
&$32$           &6.758e-06 &3.98 &6.141e-06   &3.97     &9.420e-05        &3.13     &2.990e-04  &2.00 \\           \hline
\end{tabular}
\end{center}
}
\end{table}

Table \ref{Example9:square-rec:test3-3-2-0-1} presents some numerical results for simplified $C^{-1}$-$P_2(T)/P_2(\pa T)/P_1(\pa T)/P_0(T)$ element on the uniform rectangular partition of the domain $\O_1=(0,1)^2$ when the exact solution is given by $u=(x^2+y^2)^{0.8}$. The parameters are $\gamma_1=\gamma_2=\gamma_3=1$; the drift vector is $\pmb{\mu}=[0,0]'$; and the diffusion tensor $a=\{a_{ij}\}$ is given by $a_{11}=1+\frac{x^2}{x^2+y^2}$, $a_{12}=a_{21}=\frac{xy}{x^2+y^2}$, $a_{22}=1+\frac{y^2}{x^2+y^2}$.
Note that the diffusion tensor $a(x)$ fails to be continuous at the corner $(0, 0)$. The numerical results in Table \ref{Example9:square-rec:test3-3-2-0-1} indicate that the convergence rate for $e_h$ in the $L^2$ norm seems to be of an order higher than ${\cal O}(h)$, for which no theory has been developed  in this paper.

\begin{table}[htbp]\centering\scriptsize
\caption{Convergence rates for simplified $C^{-1}$-$P_2(T)/P_2(\pa T)/P_1(\pa T)/P_0(T)$ element on the domain $\O_1$; uniform rectangular partition; the diffusion tensor $a_{11}=1+\frac{x^2}{x^2+y^2}$, $a_{12}=a_{21}=\frac{xy}{x^2+y^2}$, $a_{22}=1+\frac{y^2}{x^2+y^2}$; the drift vector $\pmb{\mu}=[0,0]'$; the exact solution $u=(x^2+y^2)^{0.8}$; the parameters $\gamma_1=\gamma_2=\gamma_3=1$.}\label{Example9:square-rec:test3-3-2-0-1}
{
\setlength{\extrarowheight}{1.5pt}
\begin{center}
\begin{tabular}{|l|l|l|l|l|l|l|l|l|}
\hline $1/h$ &$\| \epsilon_0\|_0$& $Rate$ & $\| \epsilon_b\|_0$  &$Rate$ &    $\3bar\epsilon_n\3bar_1$& $Rate$ & $\|e_h\|_0$& $Rate$ \\
\hline
$4$     &9.962e-03  &  &2.453e-02   &   &1.490e-01     &   &1.976e-02 &\\
$8$     &2.157e-03  &2.21  &5.470e-03   &2.17   &3.789e-02     &1.98   &8.296e-03 &1.25\\
$16$    &5.178e-04  &2.06  &1.319e-03   &2.05   &9.493e-03     &2.00   &3.265e-03 &1.35\\
$32$    &1.286e-04  &2.01  &3.279e-04   &2.01   &2.375e-03     &2.00   &1.224e-03 &1.42  \\
$64$    &3.218e-05  &2.00  &8.205e-05   &2.00   &5.939e-04     &2.00   &4.455e-04 &1.46\\               \hline
\end{tabular}
\end{center}
}
\end{table}

 Table \ref{Example10:0:square-shaped-TRI:test2-2-0} demonstrates the numerical performance of S-PDWG algorithm \eqref{PDWG-scheme1}-\eqref{PDWG-scheme2} for the test problem \eqref{model-problem} when the exact solution is given by $u=\cos(x)\sin(y)$ on the uniform triangular partition of the unit square domain $\O_1=(0,1)^2$. The parameters are $\gamma_1=\gamma_2=\gamma_3=1$. The diffusion tensor $a=\{a_{ij}\}$ is given by $a_{11}=1+x$, $a_{12}=a_{21}=0.5x^{\frac{1}{3}}y^{\frac{1}{3}}$, $a_{22}=1+y$; and the drift vector is $\pmb{\mu}=[e^{1-x},e^{xy}]'$. Note that the first order derivative of the diffusion tensor $a_{12}=a_{21}$ fails to be continuous at the corner $(0, 0)$. Therefore, the theory developed in this paper is not available for this test case. However, we observe from  Table \ref{Example10:0:square-shaped-TRI:test2-2-0} that the convergence rate for $e_h$ in the $L^2$ norm seems to be of an order higher than ${\cal O}(h)$ for simplified $C^{-1}$-$P_k(T)/P_k(\pa T)/P_{k-1}(\pa T)/P_{k-1}(T)$ element with $k=1$ and $k=2$, respectively.

\begin{table}[htbp]\centering\scriptsize
\caption{Convergence rates for simplified $C^{-1}$-$P_k(T)/P_k(\pa T)/P_{k-1}(\pa T)/P_{k-1}(T)$ element on $\O_1$;  uniform triangular partition;  the exact solution $u=\cos(x)\sin(y)$; the diffusion tensor $a_{11}=1+x$,  $a_{12}=a_{21}=0.5x^{\frac{1}{3}}y^{\frac{1}{3}}$, $a_{22}=1+y$; the drift vector $\pmb{\mu}=[e^{1-x},e^{xy}]'$; the parameters $\gamma_1=\gamma_2=\gamma_3=1$.}\label{Example10:0:square-shaped-TRI:test2-2-0}
{
\setlength{\extrarowheight}{1.5pt}
\begin{center}
\begin{tabular}{|l|l|l|l|l|l|l|l|l|l|l|}
\hline &$1/h$ &$\| \epsilon_0\|_0$& $Rate$ & $\| \epsilon_b\|_0$  &$Rate$ &    $\3bar\epsilon_n\3bar_1$& $Rate$ & $\|e_h\|_0$& $Rate$ \\
\hline
&$2$            &2.867e-02 &  &4.281e-02   &     &1.502e-01        &    &4.348e-02  &  \\
$k=1$&$4$       &5.658e-03 &2.34  &8.227e-03   &2.38     &5.383e-02        &1.48     &1.869e-02  &1.22 \\
&$8$            &1.266e-03 &2.16  &1.889e-03   &2.12     &1.939e-02        &1.47     &9.576e-03  &0.96\\
&$16$           &2.635e-04 &2.26  &4.094e-04   &2.21     &5.904e-03        &1.72     &4.079e-03  &1.23 \\
&$32$           &5.685e-05 &2.21  &9.106e-05   &2.17     &1.603e-03        &1.88     &1.768e-03  &1.30 \\
\hline &$1/h$ &$\| \epsilon_0\|_0$& $Rate$ & $\| \epsilon_b\|_0$  &$Rate$ &    $\3bar\epsilon_n\3bar_1$& $Rate$ & $\|e_h\|_0$& $Rate$ \\
\hline
&$2$            &5.422e-03 & &6.239e-03   &     &3.212e-02        &     &8.392e-02  &  \\
$k=2$&$4$       &3.852e-04 &3.82 &5.474e-04   &3.51     &4.793e-03        &2.74     &2.067e-02  &2.02 \\
&$8$            &2.605e-05 &3.89 &4.175e-05   &3.71     &6.095e-04        &2.98     &5.115e-03  &2.01\\
&$16$           &1.822e-06 &3.84 &3.494e-06   &3.58     &9.340e-05        &2.71     &1.411e-03  &1.86 \\
&$32$           &1.570e-07 &3.54 &3.666e-07   &3.25     &1.822e-05        &2.36     &5.046e-04  &1.48 \\           \hline
\end{tabular}
\end{center}
}
\end{table}

\subsubsection{Numerical experiments for maximum principle}
 This numerical experiment is used to test the maximum principle for the S-PDWG method \eqref{PDWG-scheme1}-\eqref{PDWG-scheme2}. The configuration of this test problem is as follows: the domain is the unit square $\Omega=(0, 1)^2$; the exact solution is $u=-x(x-1)y(y-1)$; the diffusion tensor is $a(x)=I$ where $I$ is an identity matrix; and the drift vector is $\pmb{\mu}=[0,0]'$. This numerical experiment was also tested in \cite{maxWG_NMPDE2019}. We observe from Table \ref{Example4:0:square-REC:test4-2-0} that the convergence rate for $e_h$ in the $L^2$ norm is of an expected optimal order ${\cal O}(h^{2})$ for the simplified $C^{-1}$-$P_2(T)/P_2(\pa T)/P_{1}(\pa T)/P_1(T)$ element on the uniform triangular partitions; and the convergence rate for $e_h$ in the $L^2$ norm seems to be of a superconvergence order ${\cal O}(h^{2})$ for the simplified $C^{-1}$-$P_2(T)/P_2(\pa T)/P_{1}(\pa T)/P_0(T)$ element on uniform square partitions, compared with the expected optimal order ${\cal O}(h)$. 

\begin{table}[htbp]\centering\scriptsize
\caption{Convergence rates for simplified $C^{-1}-P_2(T)/P_2(\pa T)/P_{1}(\pa T)/P_s(T)$ element with exact solution $u=-x(x-1)y(y-1)$ on $\O_1$; uniform triangular partitions for $s=1$, uniform square partitions for $s=0$; the diffusion tensor $a=I$; the drift vector $\pmb{\mu}=[0,0]'$; the parameters $\gamma_1=\gamma_2=\gamma_3=0$.}\label{Example4:0:square-REC:test4-2-0}
{
\setlength{\extrarowheight}{1.5pt}
\begin{center}
\begin{tabular}{|l|l|l|l|l|l|l|l|l|l|}
\hline &$1/h$ &$\| \epsilon_0\|_0$& $Rate$ & $\| \epsilon_b\|_0$  &$Rate$ &    $\3bar\epsilon_n\3bar_1$& $Rate$ & $\|e_h\|_0$& $Rate$ \\
\hline
&$2$        &2.524e-03 & &1.940e-03   &   &8.254e-03     &   &3.969e-02 &\\
$s=1$&$4$      &1.784e-04 &3.82  &1.445e-04   &3.75   &1.014e-03     &3.03   &1.126e-02 &1.82\\
     &$8$      &1.181e-05 &3.92  &9.851e-06   &3.87   &1.243e-04     &3.03   &2.858e-03 &1.98\\
&$16$          &7.546e-07 &3.97  &6.377e-07   &3.95   &1.526e-05     &3.03   &7.153e-04 &2.00  \\
&$32$          &4.753e-08 &3.99  &4.034e-08   &3.98   &1.875e-06     &3.03   &1.789e-04 &2.00\\
 \hline &$1/h$ &$\| \epsilon_0\|_0$& $Rate$ & $\| \epsilon_b\|_0$  &$Rate$ &    $\3bar\epsilon_n\3bar_1$& $Rate$ & $\|e_h\|_0$& $Rate$ \\
\hline
 &$4$         &9.976e-04 & &2.311e-03   &  &1.083e-03     &   &5.545e-04 &\\
$s=0$&$8$       &2.414e-04 &2.05  &5.834e-04   &1.99   &1.403e-04     &2.95   &2.040e-04 &1.44\\
&$16$           &5.985e-05 &2.01  &1.461e-04   &2.00   &1.777e-05     &2.98   &5.765e-05 &1.82\\
&$32$           &1.493e-05 &2.00  &3.654e-05   &2.00   &2.215e-06     &3.00   &1.487e-05 &1.96  \\
&$64$           &3.731e-06 &2.00  &9.136e-06   &2.00   &2.754e-07     &3.01   &3.745e-06 &1.99\\               \hline
\end{tabular}
\end{center}
}
\end{table}

Tables \ref{Example4:0:square-REC:test4-2-0-8-2}-\ref{Example4:0:square-REC:test4-2-0-8-7} present the maximum and minimum values for the numerical approximation $u_h$ on the uniform triangular partition with the simplified $C^{-1}-P_2(T)/P_2(\pa T)/P_{1}(\pa T)/P_1(T)$ element and on the uniform square partition with the simplified $C^{-1}-P_2(T)/P_2(\pa T)/P_{1}(\pa T)/P_0(T)$ element. The sets of parameters $(\gamma_1,\gamma_2,\gamma_3)=(0,0,0)$ and $(\gamma_1,\gamma_2,\gamma_3)=(0,1,0)$ are chosen respectively. Note that $f=x(x-1)+y(y-1)\leq0$ in $\O_1$. 
Denoted by $\max_{\O_1} u_h|_v$, $\max_{\O_1} u_h|_c$ and $\max_{\O_1} u_h|_e$ the maximum values of $u_h$ at the vertexes, the centers and  the midpoints of the edges throughout all the elements $ \cup_{T\in {\cal T}_h} T \setminus \partial \Omega_1$, respectively.  
The same calculation applies to $\min_{\O_1} u_h|_v$, $\min_{\O_1} u_h|_c$, $\min_{\O_1} u_h|_e$, and $\max_{\pa\O_1} u_h|_v$, $\max_{\pa\O_1} u_h|_e$. It can be seen from Tables \ref{Example4:0:square-REC:test4-2-0-8-2}-\ref{Example4:0:square-REC:test4-2-0-8-7} that the values of $u_h$ are in the range of $(-0.0625,0)$ and $\max_{\O_1}u_h|_*<\max_{\pa\O_1}u_h|_*$, where $*$ could be $v, c, e$. This indicates that the numerical solution $u_h$ arising from S-PDWG scheme \eqref{PDWG-scheme1}-\eqref{PDWG-scheme2} satisfies the maximum principle.

\begin{table}[htbp]\centering\scriptsize
\caption{The maximum and minimum values for $u_h$ with the simplified $C^{-1}-P_2(T)/P_2(\pa T)/P_{1}(\pa T)/P_1(T)$ element on uniform triangular partition;  the parameters $\gamma_1=\gamma_3=0$.}\label{Example4:0:square-REC:test4-2-0-8-2}
{
\setlength{\extrarowheight}{1.5pt}
\begin{center}
\begin{tabular}{|l|l|l|l|l|l|l|l|}
\hline 
meshsize  &$h=\frac{1}{8}$&$h=\frac{1}{16}$ &$h=\frac{1}{32}$ &$h=\frac{1}{8}$&$h=\frac{1}{16}$&$h=\frac{1}{32}$
\\
\hline
parameter&$\gamma_2=0$&&& $\gamma_2=1$ &&
\\
\hline$\max_{\O_1}u_h|_v$&-3.342e-03&-8.876e-04&-2.288e-04& -3.336e-03&-8.872e-04&-2.287e-04\\
$\min_{\O_1} u_h|_v$&-6.418e-02&-6.293e-02&-6.261e-02& -6.418e-02&-6.293e-02&-6.261e-02\\
$\max_{\O_1} u_h|_c$&-5.379e-03 &-1.482e-03&-3.885e-04& -5.373e-03&-1.481e-03&-3.885e-04\\
$\min_{\O_1} u_h|_c$&-6.124e-02&-6.218e-02&-6.242e-02& -6.124e-02&-6.218e-02&-6.242e-02\\
$\max_{\O_1} u_h|_e$&-1.499e-04&-2.438e-05&-4.402e-06& -1.417e-04&-2.383e-05&-4.366e-06\\
$\min_{\O_1} u_h|_e$&-6.168e-02&-6.229e-02&-6.245e-02& -6.168e-02&-6.229e-02&-6.245e-02\\
$\max_{\pa\O_1} u_h|_v$&3.042e-03&8.389e-04&2.200e-04& 3.053e-03&8.395e-04&2.200e-04\\
$\max_{\pa\O_1}u_h|_e$&-1.499e-04&-1.420e-05&-4.142e-07& -1.417e-04&-1.466e-05&-4.468e-07\\
 \hline
\end{tabular}
\end{center}
}
\end{table}

\begin{table}[htbp]\centering\scriptsize
\caption{The maximum and minimum values for $u_h$ with the simplified  $C^{-1}-P_2(T)/P_2(\pa T)/P_{1}(\pa T)/P_0(T)$ element on uniform square partition;  the parameters $\gamma_1=\gamma_3=0$.}\label{Example4:0:square-REC:test4-2-0-8-7}
{
\setlength{\extrarowheight}{1.5pt}
\begin{center}
\begin{tabular}{|l|l|l|l|l|l|l|l|}
\hline meshsize &$h=\frac{1}{8}$&$h=\frac{1}{16}$ &$h=\frac{1}{32}$& $h=\frac{1}{8}$&$h=\frac{1}{16}$&$h=\frac{1}{32}$\\
\hline
parameter &$\gamma_2=0$&& &$\gamma_2=1$ &&\\
\hline
$\max_{\O_1} u_h|_v$&-7.206e-03&-1.992e-03&-5.232e-04 &-7.206e-03 &-1.993e-03&-5.233e-04\\
$\min_{\O_1} u_h|_v$&-6.213e-02&-6.241e-02&-6.248e-02 &-6.207e-02&-6.239e-02&-6.247e-02\\
$\max_{\pa\O_1} u_h|_v$&-9.030e-04&-2.346e-04&-5.981e-05 &-9.043e-04&-2.347e-04&-5.982e-05\\ \hline
\end{tabular}
\end{center}
}
\end{table}

In conclusion, the numerical performance is consistent with or better than what the theory predicts. In particular, most of the numerical results indicate a superconvergence order of error estimate especially for $s=0$ on the rectangular meshes. The last numerical test demonstrates the maximum principle holds true for the S-PDWG scheme  \eqref{PDWG-scheme1}-\eqref{PDWG-scheme2}, for which the numerical analysis will be our future work.

\section*{Acknowledgement}
 We would like to express our gratitude to Dr. Junping Wang (NSF) for his valuable discussion and suggestions.


\newpage

\begin{thebibliography}{99}
\bibitem{babuska} {\sc I. Babu\u{s}ka}, {\em
The finite element method with Lagrangian multipliers}, Numer. Math., vol. 20, pp. 179-192, 1973.

\bibitem{Brezzi-1974}  {\sc F. Brezzi}, {\em
On the existence, uniqueness and approximation of saddle point problems arising from Lagrangian multipliers}, Rev. Francaise Automat. Informat. Rrcherche Op\'{e}rationnelle S\'{e}r. Rounge., vol. 8 (R-2), pp. 129-151, 1974.

 \bibitem{Burman2013} {\sc E. Burman}, {\em Stabilized finite element methods for nonsymmetric, noncoercive, and ill-posed problems. Part I: Elliptic equations}, SIAM J. Sci. Comput., vol. 35, pp. 2752-2780, 2013.

\bibitem{burman2014} {\sc E. Burman},
{\em Stabilized finite element methods for nonsymmetric,
noncoercive, and ill-possed problems. Part II: hyperbolic
equations}, SIAM J. Sci. Comput, vol. 36, No. 4, pp.
A1911-A1936, 2014. 


\bibitem{BS1968} {\sc R. G. Bhandari and R. E. Sherrer}, {\em
Random vibrations in discrete nonlinear dynamic systems}, J. Mech. Eng. Sci., vol. 10, pp. 168-174, 1968.

\bibitem{BW1996} {\sc L. A. Bergman, S. F. Wojtkiewicz, E. A. Johnson, and B. F. Spencer, Jr.,} {\em
Robust numerical solution of the Fokker-Planck Equation for second order dynamical systems under parametric and external white noise excitations}, Fields Institute Communications., vol. 9, 1996.

\bibitem{Ciarlet-2002}  {\sc P. G. Ciarlet}, {\em
The Finite Element Method for Elliptic Problems}, Classics Appl. Math., SIAM, Philadelphia, 2002.

\bibitem{Fokker-1914}  {\sc A. D. Fokker}, {\em
Die mittlere energie rotierender elektrischer dipole im strahlungsfeld}, Ann. Phys., vol. 348, pp. 810-820, 1914.

\bibitem{Gardiner-1985}  {\sc C. W. Gardiner}, {\em
Handbook of stochastic methods}, 2nd ed., Springer-Verlag, Berlin, Heidelberg, 1985.

\bibitem{Gilbarg-1983}  {\sc D. Gilbarg and N. S. Trudinger}, {\em
Elliptic Partial Differential Equations of Second Order}, Springer-Verlag, Berlin, second edition., 1983.


\bibitem{KN-2006}  {\sc P. Kumar and S. Narayana}, {\em
Solution of Fokker-Planck equation by finite element and finite difference methods for nonlinear systems}, Sadhana., vol. 31, Part 4, pp. 445-461, 2006.
2006.

\bibitem{LW-divcurl-2020} {\sc Y. Liu and J. Wang}, {\em
A primal-dual weak Galerkin method for div-curl systems with low-regularity solutions}, https://arxiv.org/pdf/2003.11795v1.pdf.


\bibitem{maxWG_NMPDE2019} {\sc Y. Liu and J. Wang}, {\em
A discrete maximum principle for the weak Galerkin finite element method on nonuniform rectangular partitions}, Numer. Methods Partial Differ. Equ., vol. 36, pp. 552-578, 2020.

\bibitem{Lwwhyperbolic} {\sc D. Li, C. Wang, and J. Wang}, {\em A primal-dual weak Galerkin finite element method for linear convection equations in
non-divergence form}, arxiv. 1910.14073.

\bibitem{f-p_DG2016} {\sc H. Liu and Z. Wang}, {\em
An entropy satisfying discontinuous Galerkin method for nonlinear Fokker-Planck equations}, J. Sci. Comput., vol. 68 (3), pp. 1217-1240, 2016.

\bibitem{f-p_FEM1991} {\sc H. P. Langtangen}, {\em
A general numerical solution method for Fokker-Planck equations with applications to structural reliability}, Prob. Eng. Mech., vol. 6 (1), pp. 33-48, 1991.

\bibitem{f-p_DG1985} {\sc R. S. Langley}, {\em
A finite element method for the statistics of non-linear random vibration}, J. Sound Vib., vol. 101(1), pp. 41-54, 1985.

\bibitem{ellip_MWY2015} {\sc L. Mu, J. Wang and X. Ye}, {\em
A weak Galerkin finite element method with polynomial reduction}, J. Comput. Appl. Math., vol. 285, pp. 45-58, 2015.

\bibitem{ellip_MWYZ2015} {\sc L. Mu, J. Wang, X. Ye and S. Zhang}, {\em
A weak Galerkin finite element method for the Maxwell equations}, J. Sci. Comput., vol. 65, pp. 363-386, 2015.

\bibitem{ellip_MWY2017} {\sc L. Mu, J. Wang and X. Ye}, {\em
Effective implementation of the weak Galerkin finite element methods for the biharmonic equation}, Comput. Math. Appl., vol. 74, pp. 1215-1222, 2017.

\bibitem{MS-2018}  {\sc H. Mizerov\'{a} and B. She}, {\em
A conservative scheme for the Fokker-Planck equation with applications to viscoelastic polymeric fluids}, J. Comput. Phy., vol. 374, pp. 941-953, 2018.

\bibitem{Perthame2007} {\sc B. Perthame}, {\em
Transport Equations in Biology}, Frontiers in Mathematics, Birkh auser Verlag, Basel., 2007.

\bibitem{F-p_PZ2018} {\sc L. Pareschi and M. Zanella}, {\em
Structure preserving schemes for nonlinear Fokker-Planck equations and applications}, J. Sci. Comput., vol. 74, pp. 1575-1600, 2018.

\bibitem{Risken1989} {\sc H. Risken}, {\em
The Fokker-Planck equation: methods of solution and applications}, 2nd ed, Springer-Verlag, Math. Sci. Eng., vol. 60, 1989.

\bibitem{MawellLI} {\sc S. Shields, J. Li and E. Machorro}, {\em
Weak Galerkin methods for time-dependent Maxwell's equations}, Comput. Math. Appl., vol. 74, pp. 2106-2124, 2017.

\bibitem{Talenti-1965} {\sc G. Talenti}, {\em
Sopra una classe di equazioni ellittiche a coefficienti misurabili}, Ann. Mat. Pura. Applic., vol. 69, pp. 285-304, 1965.

\bibitem{WY-ellip_MC2014}  {\sc J. Wang and X. Ye}, {\em
A weak Galerkin mixed finite element method for second-order elliptic problems}, Math. Comput., vol. 83, pp. 2101-2126,  2014.

\bibitem{WW-fp-2018} {\sc C. Wang and J. Wang}, {\em
A primal-dual weak Galerkin finite element method for Fokker-Planck type equations}, https://arxiv.org/pdf/1704.05606.pdf, SIAM J. Numer. Anal, In Press.

\bibitem{wwhyperbolic} {\sc C. Wang and J. Wang}, {\em
A primal-dual finite element method for first-order transport problems}, arxiv. 1906.07336.

\bibitem{wybasis} {\sc J. Wang and X. Ye}, {\em
The basics of weak Galerkin finite element methods}, arxiv. 1901.10035v1.

\bibitem{ellip_JCAM2013} {\sc J. Wang and X. Ye}, {\em
A weak Galerkin finite element method for second-order elliptic problems}, J. Comput. Appl. Math., vol. 241, pp. 103-115, 2013.

\bibitem{ellipCau_WW2020} {\sc C. Wang and J. Wang}, {\em
Primal-dual weak Galerkin finite element methods for elliptic Cauchy problems}, Comput. Math. Appl., vol. 79, pp. 746-763, 2020.

\bibitem{PDWG_MC2017} {\sc C. Wang and J. Wang}, {\em
A primal-dual weak Galerkin finite element method for second order elliptic equations in non-divergence form}, Math. Comput., vol. 87, pp. 515-545, 2018.

\bibitem{WW_bihar-2014} {\sc C. Wang and J. Wang}, {\em
An efficient numerical scheme for the biharmonic equation by weak Galerkin finite element methods on polygonal or polyhedral meshes}, Comput. Math. Appl., vol. 68, pp. 2314-2330, 2014.

\bibitem{WW_divcur-2016} {\sc C. Wang and J. Wang}, {\em
Discretization of div-curl systems by weak Galerkin finite element methods on polyhedral partitions}, J. Sci. Comp., vol. 68, pp. 1144-1171, 2016.

\bibitem{WW_HWG-2015} {\sc C. Wang and J. Wang}, {\em
A hybridized weak Galerkin finite element method for the biharmonic equation}, Int. J. Numer. Anal. Model., vol. 12, pp. 302-317, 2015.

\bibitem{WY_stokes-2016} {\sc J. Wang and X. Ye}, {\em
A weak Galerkin finite element method for the stokes equations}, Adv. Comput. Math., vol. 42, pp. 155-174, 2016.

\bibitem{ellip_XYZ2020} {\sc X. Ye and S. Zhang}, {\em
A stabilizer-free weak Galerkin finite element method on polytopal meshes}, J. Comput. Appl. Math., In Press.

\bibitem{WY_ZZ-2015} {\sc R. Zhang and Q. Zhai}, {\em
A weak Galerkin finite element scheme for the biharmonic equations by using polynomials of reduced order}, J. Sci. Comput., vol. 64, pp. 559-585, 2015.


\end{thebibliography}
\end{document}